\documentclass[matann2]{svjour}

\usepackage{amsmath}
\usepackage{amsfonts}
\usepackage{latexsym}
\usepackage{mathrsfs}
\usepackage{times}
\usepackage{array}
\usepackage{amscd}
\usepackage{a4}
\usepackage[mathcal]{euscript}
\usepackage{amssymb}
\usepackage{epsf, epic, eepic} 
\usepackage{pstricks}
\usepackage{amsrefs} 
\usepackage{theorem}

\newtheorem{maintheorem}{Main Theorem.}

\input xy 
\xyoption{all} 
\input{boxedeps} 
\SetEPSFDirectory{} 
\HideDisplacementBoxes
\SetRokickiEPSFSpecial  
\newcommand{\bB}{\mathbb{B}} 

\newcommand{\cB}{\mathcal{B}} 
\newcommand{\cC}{\mathcal{C}} 
\newcommand{\cD}{\mathcal{D}} 
\newcommand{\cF}{\mathcal{F}} 
\newcommand{\cK}{\mathcal{K}} 
 
\newcommand{\cM}{\mathcal{M}} 
\newcommand{\cP}{\mathcal{P}} 
\newcommand{\cS}{\mathcal{S}}

\DeclareMathOperator{\Hom}{Hom}

\newcommand{\ra}{\rightarrow}

\newcommand{\bfx}{{\bf x}}
\newcommand{\cpr}{\cC\cP_R}
\newcommand{\spf}[1]{\cS_{#1}(P,\cF)}
\newcommand{\hpc}[1]{H_{#1}(P,\cF)}
\newcommand{\cpc}[1]{\cC_{#1}(P,\cF)}
\newcommand{\kbc}[1]{\cK_{#1}(\bB,\cF)}
\newcommand{\cbc}[1]{\cC_{#1}(\bB,\cF)}

\newcommand{\rmod}{\cM\text{od}_R}
\newcommand{\grrmod}{Gr\cM\text{od}_R}
\newcommand{\chr}{\text{Ch}_R}
\DeclareMathAlphabet{\ams}{U}{msb}{m}{n}
\DeclareMathAlphabet{\goth}{U}{euf}{m}{n}

\def\sl{\text{SL}}

\def\id{\text{id}}

\def\rk{\text{rk}}

\def\wtl{\widetilde}
\def\wh{\widehat}

\def\FF{\mathcal F}

\def\CC{\mathcal C}

\def\SS{\mathcal S}
\def\KK{\mathcal K}
\def\aa{\alpha}

\def\bb{\beta}
\def\ss{\sigma}

\def\ve{\varepsilon}

\def\wh{\widehat}
\def\Z{\ams{Z}}

\def\B{\ams{B}}

\def\w{\mathbf{w}}

\def\x{\mathbf{x}}

\def\0{\mathbf{0}}
\def\quo{/\kern -.45em\sim}
\newpsobject{showgrid}{psgrid}
            {subgriddiv=1,griddots=10,gridlabels=6pt,gridcolor=red}
\def\Langle{\langle\kern -2pt\langle}
\def\Rangle{\rangle\kern -1.9pt\rangle}
\title{Homology of coloured posets:  a generalisation of Khovanov's 
cube construction}

\author{Brent Everitt and
Paul Turner
\thanks{The first author was partially supported by the London and
Edinburgh Mathematical Societies, and is grateful to the Institute for
Geometry and its Applications, 
University of Adelaide, Australia, for their hospitality during an extended
visit.  The second author was partially supported by the Royal Society
and is grateful to the Glenelg Maths Institute for their hospitality.}
}

\institute{{\sc Brent Everitt:}
Department of Mathematics, University of York, York
YO10 5DD, United Kingdom. \email{bje1@york.ac.uk}. 
\hspace{1em}{\sc Paul Turner:} School of Mathematical and Computer Sciences,
Heriot-Watt University, Edinburgh, EH1 4AS, United Kingdom and D\'epartement de math\'ematiques, 
Universit\'e de Fribourg, CH-1700 Fribourg, Switzerland.
\email{paul@ma.hw.ac.uk}.
}

\titlerunning{}
\authorrunning{Brent Everitt and Paul Turner}

\begin{document}

\maketitle


\begin{abstract}
We define a homology theory for a certain class of posets equipped
with a representation. We show that when restricted to Boolean
lattices this homology is isomorphic to the homology of the ``cube''
complex defined by Khovanov.
\end{abstract}


\section*{Introduction}

Given a representation of a Boolean lattice one can construct a chain
complex by using Khovanov's ``cube'' construction in his celebrated paper on
the categorification of the Jones polynomial \cite{Khovanov00}. Recall
that a representation of a Boolean lattice assigns to each vertex $x$
a finite dimensional vector space $V_x$ and to each edge $x\leq y$ a
linear map $V_x \ra V_y$. The homology of complexes arising in this
way plays a central role in link homology theories such as Khovanov
homology and Khovanov-Rozansky homology.  Recently, Heegaard-Floer knot
homology has also been interpreted in terms of the homology of a the
complex coming from a Boolean lattice equipped with a particular
representation. Khovanov's construction relies on specific properties of
Boolean lattices and the question that motivates the
current paper is: can one define a homology theory for a more general
class of posets equipped with a representation, which
for Boolean lattices gives the homology arising from Khovanov's cube complex?

Indeed one can: we define a chain complex for an arbitrary poset with
$1$ equipped with a representation. We refer to such posets as
{\em coloured posets} which form the objects of a category and by passing
to homology we get a functor to
graded modules. This functor satisfactorily answers the above question: 
for coloured Boolean lattices the result is isomorphic to  the homology of 
Khovanov's cube complex, an outcome not \emph{a priori\/} obvious.
 
We begin in Section 1 by studying the category of coloured posets,
$\cpr$ over a ring $R$. We
provide a number of examples and several basic constructions. In
Section 2 we define a functor $\cS_*$ from $\cpr$ to chain complexes
over $R$ which generalises the well-known order homology of a poset to
the situation where one has a local system of coefficients. The
resulting homology $H_*(P,\cF)$ is what we refer to as the homology of the 
coloured poset $(P,\cF)$. We show that the chain complex $\cS_*(P,\cF)$ is 
homotopy equivalent to a much smaller complex $\cC_*(P,\cF)$, paralleling 
the situation  in topology where the full simplicial chain complex on a space
is cut down by throwing away degeneracies.

The main technical result is presented in Section 3 where we show that
a coloured poset obtained by gluing two coloured posets together by a morphism
gives rise to a long exact 
sequence in homology (see Theorem \ref{section:les:result100}).
We give a brief tutorial in Section 4 on Khovanov's cube complex,
which in the context of this paper is a chain complex $\KK_*(\B,\FF)$
associated to a coloured Boolean lattice $(\B, \FF)$. We denote its
homology by $H_{*}^\diamond(\B,\FF)$. In Section 5 we present
the main result, namely the agreement of the coloured poset homology
with the Khovanov's cube homology for coloured Boolean lattices. We construct a
chain map $\phi$ from the cube complex  $\KK_*(\B,\FF)$ to $\CC_*(\B,\FF)$,
and our main result, given as Theorem \ref{thm:main} in \S\ref{main:theorem}, 
is

\begin{maintheorem}
Let $(\B,\FF)$ be a coloured Boolean lattice. 
Then $\phi:\KK_*(\B,\FF)\ra\CC_*(\B,\FF)$ 
is a quasi-isomorphism, yielding isomorphisms,
$$
\xymatrix{H_{n}^\diamond(\B,\FF) \ar[r]^-{\cong} & H_{n}(\B,\FF)}.
$$
\end{maintheorem}

\section{Coloured posets}\label{section:colouredposets}

The principal characters in our story, coloured posets, are partially ordered sets
(posets) whose elements are labeled by $R$-modules so
that there is a homomorphism between the labels of comparable elements. More 
concisely, a coloured poset is a 
representation of a poset with maximal element.

We begin by recalling basic poset terminology, for which we will
generally follow
\cite{Stanley97}*{Chapter 3}. A {\em poset} $(P,\leq)$ is a set $P$ together 
with a reflexive, anti-symmetric, transitive binary relation $\leq$,
and a {\em map of posets} $f:(P,\leq)\rightarrow (Q,\leq')$ is a set
map preserving the respective relations, i.e. $f(x)\leq'f(y)$ in $Q$
if $x\leq y$ in $P$.  One writes $x<y$ when $x\leq y$ and
$x\not=y$. If $x<y$ and there is no $z$ with $x<z<y$ then we say that
$y$ \emph{covers\/} $x$, and write $x<_c y$. The covering relation is
illustrated via the {\em Hasse diagram}: the graph with vertices the
elements of $P$, and an edge joining $x$ to $y$ iff $x<_c y$. We will
follow the convention that Hasse diagrams will be presented vertically on the page
with
$y$ drawn above $x$ whenever $x<_c y$.

An ordered {\em multi-sequence} is a sequence $x_1\leq\cdots\leq x_k$, 
of comparable elements. An ordered \emph{sequence} is a
multi-sequence with $x_1<\cdots<x_k$. A sequence is {\em
saturated} when it has the form $x_1<_c\cdots<_c x_k$. This differs
from the standard poset terminology (where a multi-sequence is called
a multi-chain and a sequence a chain) justified by our giving
preference to homological notions, where the term chain is already
taken.  There is the obvious notion of a $0$: an element with 
$x\geq 0$ for all $x\in P$; similarly for a $1$.
A poset $P$ is {\em graded of rank} $r$
if every saturated
sequence, maximal under inclusion of sequences,
has the same length $r$.
There is then
a unique grading or rank function $\rk:P\rightarrow\{0,1,\ldots,r\}$ with $\rk(x)=0$ 
if and only if $x$ is minimal, and 
$\rk(y)=\rk(x)+1$ whenever $x<_c y$.
The rank $1$ elements are called the {\em atoms}.

Sometimes our posets will turn out to be {\em lattices}: posets for which any $x$ and $y$
have a supremum or least upper bound $x\vee y$ (the join of $x$ and $y$)
and an infimum, or greatest lower bound $x\wedge y$ (the meet of $x$ and $y$). 
A lattice is {\em atomic} if every element can be expressed (not necessarily uniquely)
as a join of atoms.

Without
explicitly mentioning it, we will often consider a poset as a
category whose objects are the elements of the poset, and with a unique
morphism $x\ra y$ between any two comparable elements $x\leq y$.
A {\em representation} of a poset is a covariant functor to some category of
modules.

Here is a primordial example: the Boolean lattice $\B=\B(X)$ on the set $X$
is a lattice isomorphic, by a bijective poset mapping,
to the lattice of subsets of $X$ under 
inclusion. We will often suppress the isomorphism and identify the elements
of $\B$ with subsets of $X$.
If $X$ is finite, then
$\B$ is graded with $\rk(S)=|S|$ for $S\subseteq X$,
and atomic, with atoms 
the singletons.
If the atoms are given some fixed ordering $a_1,\ldots,a_r$,
then every $x\in\B$ can be expressed \emph{uniquely} as a join,
\begin{equation}\label{joinofatoms}
x=\bigvee a_{i_j}=a_{i_1}\vee a_{i_2}\vee\cdots\vee a_{i_k},
\end{equation}
where $i_1<\cdots<i_k$. 
For $ x= a_{i_1}\vee\cdots\vee a_{i_k}$, one has the covering relation
$x<_c y$ if and only if the unique expression for 
$y$ is $y=(a_{i_1}\vee\cdots\vee
a_{i_j})\vee a_\ell\vee (a_{i_{j+1}}\vee\cdots\vee a_{i_k})$.

With these preliminaries out of the way we now make the principal definition.
Fix a unital commutative ring $R$
and let $\rmod$ be the category of $R$-modules.

\begin{definition} A {\em coloured poset} $(P,\cF)$ consists of 
\begin{itemize}
\item a poset $P$ having a unique maximal element $1_P$, and
\item a covariant functor $\cF \colon P \ra \rmod$.
\end{itemize}
The functor $\cF$ will be referred to as the {\em colouring}.

A {\em morphism } of coloured posets $(P_1,\cF_1) \ra (P_2,\cF_2)$ is a pair  
$(f,\tau)$ where
\begin{itemize}
\item $f\colon P_1 \ra P_2$ is a map of posets, and
\item $\tau$ is a collection $\{\tau_x\}_{x\in P_1}$ where 
$\tau_x\colon \cF_1(x) \ra \cF_2(f(x))$ is an $R$-module homomorphism. 
\end{itemize}
This data satisfies the following two conditions
\begin{enumerate}
\item $f(x) = 1_{P_2}$ if and only if $x=1_{P_1}$, and
\item (naturality) for all $x\leq y$ in $P_1$, the following diagram commutes
\[
\xymatrix{
\cF_1(x) \ar[rr]^{\cF_1(x\leq y)} \ar[d]_{\tau_x} & &\cF_1(y) \ar[d]^{\tau_{y}}\\
\cF_2(f(x)) \ar[rr]^{\cF_2(f(x) \leq f(y))}  & &\cF_2(f(y)).
}
\]
\end{enumerate}

Coloured posets and morphisms between them form a category denoted $\cpr$.
\end{definition}

Thus the colouring
associates to each element of the poset an $R$-module, $\cF(x)$ and if
$x\leq  y$ then there is an associated map $\cF(x \leq y) \colon
\cF(x) \ra \cF(y)$. We will often find it convenient
to write $\cF_x^{y}$ instead of $\cF(x\leq y)$. Also, we usually
 define the morphisms $\FF(x\leq y)$ just in the cases where $x<_c y$,
as all the others can be recovered from these by repeated composition.

Given a map of (uncoloured) posets $f\colon P_1 \ra P_2$, the
composite $\cF_2\circ f\colon P_1 \ra \rmod$ defines another colouring
on $P_1$. Condition (2) in the above definition is merely stating that
$\tau$ is a natural transformation (of functors $P_1\ra \rmod$) from
$\cF_1$ to $\cF_2\circ f$.

The notion of a coloured poset is a rather general one encompassing
many interesting examples as the following illustrate.

\begin{example} 
Let $P$ be a poset with unique maximal element and $A$ an
$R$-module. The {\em constant colouring} on $P$ by $A$ is defined by
the colouring functor $\cF\colon P \ra \rmod$ given by $\cF(x) = A$
and $\cF_x^{y} = \id_A$ for all $x\leq y$.
\end{example}

\begin{example}[Pre-sheaves]\label{ex:sheaf}
Let $X$ be a topological space and $P$ the poset of open subsets 
partially ordered by {\em reverse} inclusion. 
A colouring is equivalent to a pre-sheaf of $R$-modules on $X$.
\end{example}

\begin{example}[abelian subgroups] Let $G$ be a group. 
Then the poset of abelian subgroups is a naturally a
coloured poset, by colouring an element of the poset with the subgroup it
corresponds to. Homomorphisms are just the inclusions.
\end{example}

\begin{example}[The Khovanov colouring]\label{ex:khovanov}
This is a colouring of a Boolean lattice associated to a link diagram.
Let $D$ be a projection of a link i.e. a link diagram,  and let
$\B$ be the Boolean lattice on the crossings of the diagram. Each
crossing can be resolved into a $0$ or a $1$-resolution as shown on
the left in Figure \ref{resolutions}. If $S$ is some subset of crossings,
then the complete resolution $D(S)$ is what results from $1$-resolving
the crossings in $S$ and $0$-resolving the crossings not in $S$: it is
a collection of planar circles.
These complete resolutions are central in Kauffman's formulation of
the Jones polynomial and also in Khovanov's definition of his homology
for links \cites{Bar-Natan02,Khovanov00}.

Now let $V$ be a graded commutative Frobenius algebra over $R$ with
multiplication $m$ and comultiplication $\mu$ both of degree $-1$. We
define a (graded) colouring $\FF\colon \B\rightarrow\grrmod$ as
follows: for $S\in\B$, let $\FF(S)=V^{\otimes k}[\text{rk}(S)]$, with
a tensor factor corresponding to each connected component of $D(S)$
shifted by the rank of $S$ in $\B$. The notation is that for $W_*$ a
graded module $(W_*[a])_i= W_{i-a}$. If $S<_c T$ in $\B$ then $D(T)$
results from $1$-resolving a crossing that was $0$-resolved in $D(S)$,
with the qualitative effect being that two of the circles in $D(S)$
fuse into one in $D(T)$, or one of the circles in $D(S)$ bifurcates
into two in $D(T)$.  In the first case $\FF(S<_c T):V^{\otimes
k}[\text{rk}(S)]\rightarrow V^{\otimes k-1}[\text{rk}(T)]$ is the map
using $m$ on the tensor factors corresponding to the fused circles,
and the identity on the others.  In the second, $\FF(S<_c
T):V^{\otimes k}[\text{rk}(S)]\rightarrow V^{\otimes
k+1}[\text{rk}(T)]$ is the map using $\mu$ on the tensor factor
corresponding to the bifurcating circles, and the identity on the
others. In both cases $\FF(S<_c T)$ is a grading preserving map. The
properties of a Frobenius algebra guarantee that $\FF$ is a
well-defined functor.

It is worth noting that (un-normalised) Khovanov homology is then
defined as the homology of a certain complex obtained from this
``cube''. In \S\ref{section:cubecomplex} we explain Khovanov's construction of this
complex, which we will call the {\em cube complex} of a coloured Boolean
lattice. For a very restrictive class of graded Frobenius algebras this results
in a bi-graded homology theory which after normalisation (depending on
an orientation) gives an invariant of oriented links.
More recent link homology theories, such as Khovanov-Rozansky homology, 
are also defined as the homology of the cube complex of a certain coloured 
Boolean lattice associated to a link diagram.
\end{example}

\begin{figure}
\begin{pspicture}(0,0)(15,3.5)
\rput(-2,-.25){
\rput(3,2){\BoxedEPSF{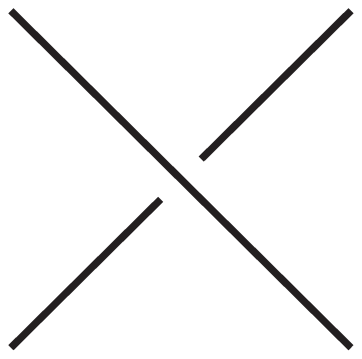 scaled 350}}
\rput(7,2){\BoxedEPSF{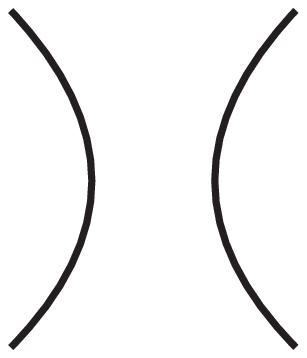 scaled 350}}
\rput(5,2){\BoxedEPSF{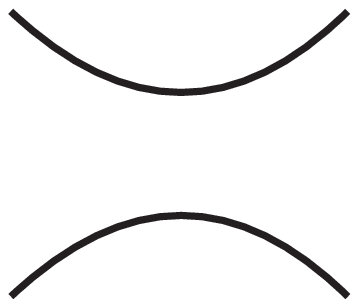 scaled 350}}
\rput(5,3.5){$0$-resolution}\rput(7,3.5){$1$-resolution}
}
\rput(5,.25){
\rput(3,2){\BoxedEPSF{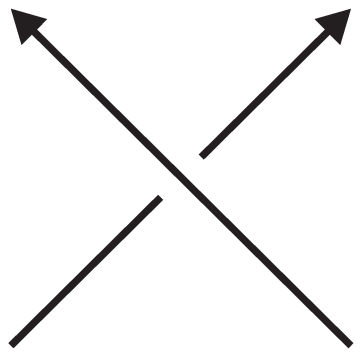 scaled 350}}
\rput(3,.5){\BoxedEPSF{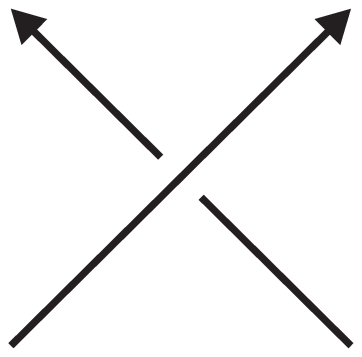 scaled 350}}
\rput(5.5,.5){\BoxedEPSF{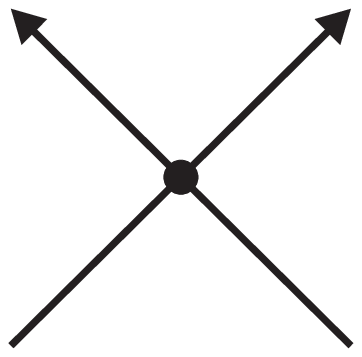 scaled 350}}
\rput(8,.5){\BoxedEPSF{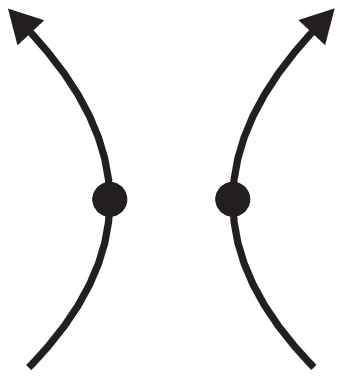 scaled 350}}
\rput(8,2){\BoxedEPSF{singularization.eps scaled 350}}
\rput(5.5,2){\BoxedEPSF{smoothing1.eps scaled 350}}
\rput(5.5,3){$0$-resolution}\rput(8,3){$1$-resolution}
\rput(3,2.55){${\red -}$}\rput(3,-.05){${\red +}$}
}
\end{pspicture}
\caption{$0$- and $1$-resolutions of a crossing in the Khovanov colouring (left)
and Ozsv{\'a}th-Szab{\'o} colouring (right).}\label{resolutions}
\end{figure}

\begin{example}[The Ozsv{\'a}th-Szab{\'o} colouring]\label{ex:os}
Although it has more geometric origins, 
knot Floer homology  now has a completely combinatorial
description involving a coloured Boolean lattice, along the lines of
Khovanov homology.
Let $\B$ be the Boolean 
lattice on the $n$ crossings of a link diagram $D$,
so that as before each crossing can be resolved
into a $0$ or a $1$-resolution, this time as shown on the right
of Figure \ref{resolutions}. 
Call the resolutions along the top row smoothings and
singularizations respectively.
If $S$ is a set of crossings, then the
complete resolution $D(S)$ is the graph resulting from
$1$-resolving the crossings in $S$ and $0$-resolving the crossings not
in $S$. Let $R=\Z[t,s_0,\ldots,s_{2n}]$, 
with the $s_i$ corresponding to the edges of a completely resolved diagram.
For $S\in\B$, let $\FF(S)$ be the quotient $A_S$ of $R$
obtained by introducing certain relations determined by $D(S)$.
If $S<_c T$ then there is a single crossing which is smoothed (resp.
singularized) in $S$ that is singularized (resp. smoothed) in $T$.
The map $\FF(S<_c T):A_S\rightarrow A_T$ is then a certain zip (resp. unzip)
homomorphism.
The precise details, which are a little more elaborate than
in the previous example, can be found in \cite{Ozsvath07}.

It turns out that $\FF$ is a colouring of the Boolean lattice associated
to an oriented knot diagram, and the homology of the associated cube complex is 
isomorphic to the Heegaard-Floer knot homology of the knot.
\end{example}

\begin{example}[The colouring of a Boolean lattice associated to a graph]
Let $\Gamma$ be a graph, $\B$ the Boolean lattice on the edge set, and
$M$ an $R$-algebra with multiplication $m$. If 
$S$ is some set of edges then let the graph $\Gamma(S)$ have the same vertex set as
$\Gamma$ and edge set $S$. Define $\FF:\B\rightarrow\cM\text{od}_R$ as follows:
 if $S\in\B$,
let $\FF(S)=M^{\otimes k}$, with a tensor factor corresponding
to each connected
component of the graph $\Gamma(S)$. If $S<_c T$ in $\B$ then $T=S\cup\{e\}$ for
some edge $e$. In particular, the graph $\Gamma(T)$ either has the same number of 
components as $\Gamma(S)$, or the edge $e$ connects two components, reducing the
overall number by one. Define $\FF_S^T=\id$ in the first case, and in the second,
$\FF_S^T:M^{\otimes k}\rightarrow M^{\otimes k-1}$ is the map using $m$
on the tensor factors corresponding to the components connected by
$e$, and the identity on the others.  This procedure gives a colouring $\FF$ of the 
Boolean lattice of a graph, first defined by Helme-Guizon and Rong
(see \cite{Helme-Guizon05}), and the homology of the associated cube complex
is related to the chromatic polynomial of $\Gamma$.
\end{example}

There are a number of interesting  constructions with coloured posets 
which we now discuss.

\subsubsection*{Unions.}
Let $(P_1,\cF_1)$ and $(P_2,\cF_2)$ be coloured
posets. Their {\em union}, $(P_1, \cF_1) \cup (P_2, \cF_2)$ is defined
by taking the disjoint union of $P_1$ and $P_2$ and then identifying
$1_{P_1}$ with $1_{P_2}$ (so the underlying poset of the union is almost,
but not quite,
the union of the underlying posets). The colouring is defined by $\cF_1$ and $\cF_2$
with the modification that $1$ is coloured by $\cF_1(1_{P_1}) \oplus
\cF_2(1_{P_2})$, and for $x\in P_1$ we have $\FF_x^1=\FF_1\oplus 0$ (and 
similarly $\FF_y^1=0\oplus\FF_2$ for $y\in P_2$).

\subsubsection*{Products.}
The {\em product} 
$(P_1,\cF_1)\times (P_2,\cF_2)=(P,\FF)$, has underlying poset $P$ the 
direct product of the $P_i$, ie: the poset with
elements $(a,b)\in P_1\times P_2$
with $(a,b) \leq (a^\prime,b^\prime)$ iff $a\leq a^\prime$ and
$b\leq b^\prime$. The colouring is $\cF(a,b) = \cF_1(a) \otimes_R
\cF_2(b)$ and $\cF_{(a,b)}^{(a^\prime,b^\prime)}= (\cF_1)_a^{a^\prime}
\otimes (\cF_2)_b^{b^\prime}$.

For example, if $\B_i$, ($i=1,2$)  are Boolean latices of rank $r_i$ (isomorphic to 
the lattice of subsets of $X_i$), then $\B_1\times\B_2$ is Boolean of
rank $r_1+r_2$ (isomorphic to the lattice of subsets of $X_1\amalg X_2$). 
If the $\B_i$ are coloured by $\FF_i$, we have a picture like Figure \ref{products},
in the case $r_1=1,r_2=2$,
and where we have abbreviated $U_x:=\FF_1(x), V_x:=\FF_2(x)$.

\begin{figure}
\begin{pspicture}(0,0)(15,4)
\rput(0.4,0){
\rput(1,2){\BoxedEPSF{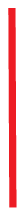 scaled 350}}
\rput(1,1.45){$U_0$}
\rput(1,2.55){$U_1$}
\rput(1,.8){$(\B_1,\FF_1)$}
}
\rput(.9,0){
\rput(4,2){\BoxedEPSF{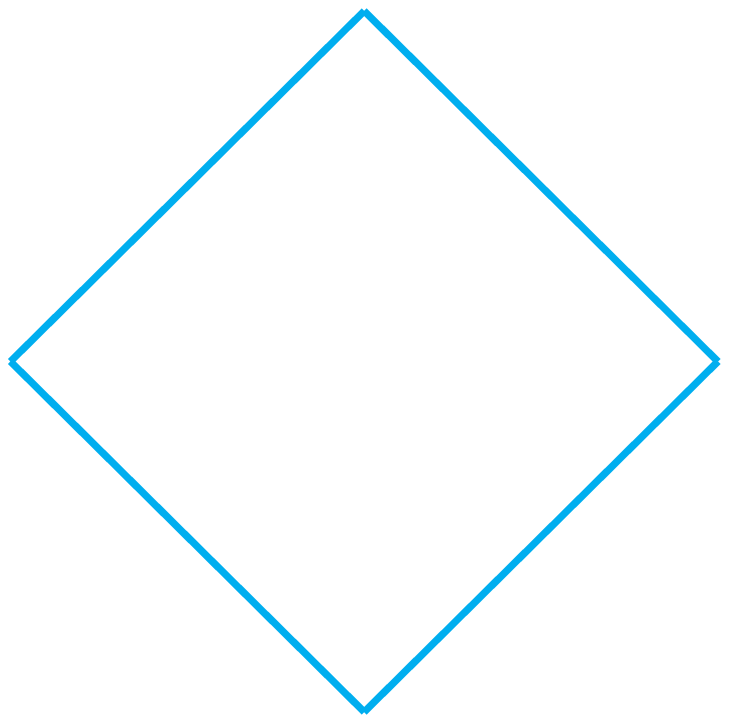 scaled 350}}
\rput*(4,0.8){$V_0$}
\rput*(2.8,2){$V_x$}
\rput*(5.2,2){$V_y$}
\rput*(4,3.2){$V_1$}
\rput(4,.3){$(\B_2,\FF_2)$}
}
\rput(2.5,0){
\rput(7.5,2){\BoxedEPSF{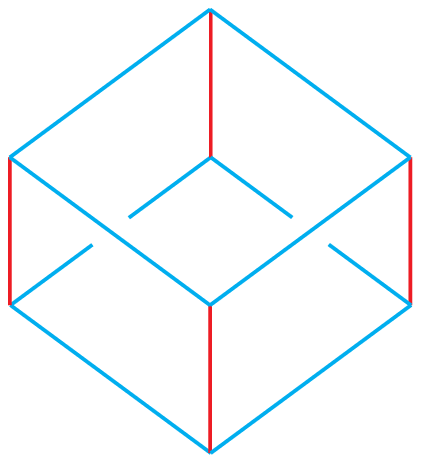 scaled 800}}
\rput*(7.5,0.15){$U_0\otimes V_0$}
\rput*(5.9,1.35){$U_0\otimes V_x$}
\rput*(7.5,1.35){$U_1\otimes V_0$}
\rput*(9.15,1.35){$U_0\otimes V_y$}
\rput*(5.9,2.6){$U_1\otimes V_x$}
\rput*(7.5,2.6){$U_0\otimes V_1$}
\rput*(9.15,2.6){$U_1\otimes V_y$}
\rput*(7.5,3.85){$U_1\otimes V_1$}
}
\rput(13.6,2){$(\B_1\times\B_2,\FF_1\otimes\FF_2)$}
\end{pspicture}
\caption{The product of coloured
Boolean lattices of ranks $1$ and $2$, yielding a coloured Boolean lattice
of rank $3$.}\label{products}
\end{figure}

\subsubsection*{Gluing along a morphism.}
Let $(P_1,\cF_1)$ and $(P_2,\cF_2)$ be
coloured posets and let $(f,\tau)\colon (P_1,\cF_1)\ra (P_2,\cF_2)$ be a
morphism of coloured posets. We can construct a new coloured poset
$(P_1,\cF_1)\cup_f (P_2,\cF_2)$ by ``gluing'' $P_1$ to $P_2$ using the map
$f$. 

The underlying set of $(P_1,\cF_1)\cup_f (P_2,\cF_2)$ is $P_1 \cup P_2$,  the 
union of elements on $P_1$ and $P_2$. The partial order on this set is defined 
as follows.
\begin{itemize}
\item If $a,a^\prime\in P_i$ then $a\leq a^\prime$ iff  $a\leq a^\prime$ in $P_i$;
\item if $a\in P_1$ and $a^\prime\in P_2$ then $a\leq a^\prime$ iff $f(a) \leq a^\prime$ 
in $P_2$.
\end{itemize}
We will denote this poset by $P_1 \cup_f P_2$.

The colouring functor $\cF\colon P_1 \cup_f P_2 \ra \rmod$ is defined as
follows. 
For an object
$a\in P_i$ set $\cF(a) = \cF_i(a)$. For a morphism $a\leq a^\prime$ we define 
$\cF_a^{a^\prime}\colon \cF(a) \ra \cF(a^\prime)$ as follows.
\begin{itemize}
\item If $a,a^\prime\in P_i$ then $\cF_a^{a^\prime} = (\cF_i)_a^{a^\prime}$, and
\item if $a\in P_1$ and $a'\in P_2$ then as part of the morphism $(f,\tau)$ there is a 
map $\tau_a\colon \cF_1(a) \ra \cF_2(f(a))$. Since $f(a) \leq a^\prime$ in
$P_2$ there is a map $(\cF_2)_{f(a)}^{a^\prime}\colon \cF_2(f(a)) \ra
\cF_2(a^\prime)$. In this case set 
\[
\cF_a^{a^\prime} =
(\cF_2)_{f(a)}^{a^\prime}\circ \tau_a.
\]
\end{itemize}

\begin{lemma}
$(P_1 \cup_f P_2, \cF)$ is a coloured poset.
\end{lemma}

\begin{proof}
It is routine to check that $P_1 \cup_f P_2$ is a poset and moreover 
that $1_{P_2}$ provides a unique maximal element.
To verify that $\cF$ is a functor the only real issue 
is composition. Suppose $a\leq a^\prime \leq a^{\prime\prime}$ then we
must check $\cF_{a^{\prime}}^{a^{\prime\prime}} \circ
\cF_{a}^{a^{\prime}} = \cF_{a}^{a^{\prime\prime}}$. 
There are a number of cases to consider. 
If $a, a^\prime \in P_1$ and $a^{\prime\prime}\in P_2$,
then the identity to check is given by the outside routes around the following diagram.
$$
\xymatrix{
\cF_1(a) \ar[r]^{(\cF_1)_{a}^{a^\prime}} \ar[dd]_{\tau_a} 
& \cF_1(a^{\prime}) \ar[r]^{\tau_{a^\prime}}
& \cF_2(f(a^\prime)) \ar[dd]^{(\cF_2)_{f(a^\prime)}^{a^{\prime\prime}}}\\
& \\
\cF_2(f(a)) \ar[rr]_{(\cF_2)_{f(a)}^{a^{\prime\prime}}} 
\ar[uurr]_{(\cF_2)_{f(a)}^{f(a^{\prime})}} 
& 
& \cF_2(a^{\prime\prime})
}
$$
The righthand triangle commutes courtesy of the functoriality of
$\cF_2$ and the lefthand triangle commutes because of the naturality
of $\tau$, thus the square commutes. 
The other cases are simpler and omitted.
\qed
\end{proof}

\begin{example}\label{example:glueings}
Let $(\bB_i, \cF_i)$ for $i=0,1$, be coloured Boolean lattices of the same rank
(so both isomorphic to the lattice of subsets of $X$) and 
$(f,\tau)\colon (\bB_0, \cF_0)\ra (\bB_1,\cF_1)$ a
morphisms of coloured posets with $f$ an isomorphism.
Then $(\bB_0, \cF_0)\cup_f (\bB_1,\cF_1)$ is
a Boolean lattice, isomorphic to the lattice of subsets of $X\cup f(0_{\B_0})$;
see Figure \ref{glueings} for the rank $|X|=2$ case.

\begin{figure}
\begin{pspicture}(0,0)(15,4)
\rput(-2.5,0){
\rput(7.5,2){\BoxedEPSF{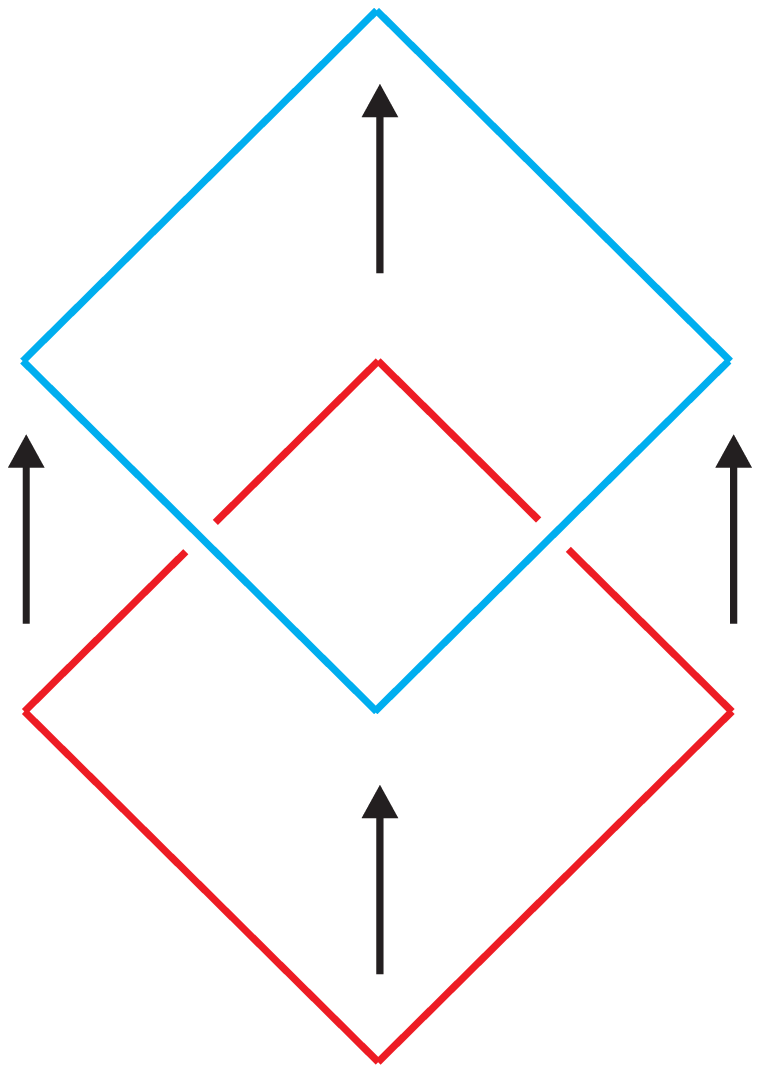 scaled 350}}
\rput*(7.5,0.25){$U_0$}
\rput*(6.25,1.35){$U_x$}
\rput*(7.5,1.35){$V_0$}
\rput*(8.75,1.35){$U_y$}
\rput*(6.25,2.6){$V_x$}
\rput*(7.5,2.6){$U_1$}
\rput*(8.75,2.6){$V_y$}
\rput*(7.5,3.85){$V_1$}
\rput(5,1){$(\B_0,\FF_0)$}
\rput(5,3){$(\B_1,\FF_1)$}
\rput(5,2){\BoxedEPSF{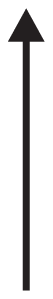 scaled 350}}
\rput(4.8,2){$f$}
\rput(9.3,2){$\left.\begin{array}{c}
\vrule width 0 mm height 36 mm depth 0 pt\end{array}\right\}$}
\rput(10.6,2){$(\B_1\bigcup_f\B_2,\FF)$}
}
\rput(-2.5,0){
\rput(13.5,3){\BoxedEPSF{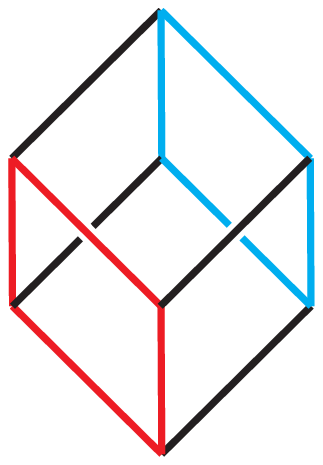 scaled 350}}
\rput(13.5,1){\BoxedEPSF{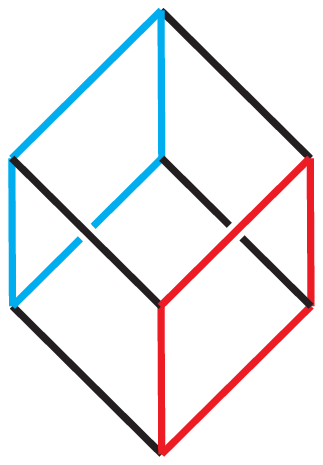 scaled 350}}
\rput(14.5,3){$\ell=3$}\rput(14.5,1){$\ell=1$}}
\end{pspicture}
\caption{Gluing coloured Boolean lattices of rank $2$ along a morphism
to give a coloured Boolean lattice of rank $3$ (left), and
the three decompositions of a rank $3$ lattice as glued
rank $2$ lattices: for $\ell=2$ (left) and $\ell=1,3$ (right).}\label{glueings}
\end{figure}

Turning it around, any coloured Boolean lattice $\B$ of rank $r$ can be decomposed 
$\B=\B_0\bigcup_f\B_1$ in a number of ways, each corresponding
to a pair of opposite faces of the ``cube'': with the atoms ordered
$a_1,\ldots,a_r$, and $a_\ell$ a fixed atom, let $\B_0$ be the subposet
consisting
of $0_\B$ and those $x\in\B$ for which the join (\ref{joinofatoms}) does not contain
$a_\ell$, 
and $\B_1$ those $x$ where it does. Then the $\B_i$ are sub-Boolean
of rank $r-1$. For $x\in\B_0$, define $f(x)=x\vee a_\ell$, and
$\tau_x:=\FF_x^{f(x)}$.
Figure \ref{glueings} illustrates 
the rank three case, with the three
decompositions (clockwise from the main picture) for $\ell=2,3$ and $1$.
The last will play a key role in \S\S\ref{section:les}-\ref{main:theorem}.
\end{example}

\section{The homology of a coloured poset}\label{section:homology}

Poset homology was pioneered by Folkman and Rota, amongst others.
We will make no attempt to summarize this vast area beyond our immediate
needs, but to whet the readers appetite we mention a couple of
fruitful applications: it provides an organizing principle in group 
representation theory, where group actions on posets lead to
representations on the poset homology
\cite{Curtis87}*{\S 66}; in the theory of hyperplane
arrangements it plays a key role, where $P$ is the intersection
lattice of the arrangement, see, eg: \cite{Orlik92}*{\S 4.5}.  The
basic principle is to pass from posets to abstract simplicial
complexes.  Recall that an abstract simplicial complex with vertex set
$X$ is a subset $\Delta\subset 2^X$ such that $\{x\}\in\Delta$ iff
$x\in X$, and
$\ss\in\Delta,\tau\subset\ss\Rightarrow\tau\in\Delta$. The
$k$-simplicies $\Delta_k$ are the $k+1$-element subsets and the empty
set $\varnothing\in\Delta$ is the unique $(-1)$-simplex.  If $P$ is a
poset then the order complex $\Delta(P)$ has $X=P$ and $k$-simplicies
the ordered sequences $\ss=(x_0<x_1<\cdots<x_k)$ of length $k+1$. If
$P$ has a $1$, then $\Delta(P)$ is a cone on
$\Delta(P\setminus 1)$, hence contractible (this is standard, but see
for example \cite{Orlik92}*{Lemma 4.96}). A similar thing is true if
$P$ has a $0$, so this procedure is normally applied to the
order complex on the poset with $0$'s and $1$'s removed: the so-called
Folkman complex.  More details on
poset topology and homology can be found in \cite{Bjorner95}. One can rephrase 
most statements in traditional poset topology in terms of classifying spaces of 
categories if one wishes.  

Our purpose in this section is to  
incorporate a colouring of $P$ into this scheme. 
Essentially this amounts to considering a local coefficient system 
(given by the colouring) on the order complex. This has already made a 
a brief
appearance in the poset literature (see \cite{Orlik92}*{\S 4.6}).
We use the
following notation: an ordered multi-sequence $x_1\leq x_2\leq
\cdots \leq x_n$ will be abbreviated to $\x=x_1x_2\cdots x_n$, and we will write
$1:=1_P$.

If $(P,\cF)$ is a coloured poset we define the chain complex
$\cS_*(P,\cF)$ by setting,
\begin{equation}\label{complex:S}
\cS_k(P,\cF) = \bigoplus_{\substack{x_1x_2\cdots x_k\\ x_i\in P\setminus 1}}
\kern-2mm\cF(x_1),
\end{equation}
for $k>0$.
Thus we have one direct summand for each length $k$ 
multi-sequence $x_1\leq x_2\leq \cdots \leq x_k$ in $P\setminus 1$. 
A typical element can thus be written as 
$\sum_{{\bf x}} \lambda\cdot{\bf x}$,
where the sum is over all length $k$ sequences 
${\bf x}= x_1x_2\cdots x_k$ and $\lambda\in \cF(x_1)$,
and when it is important to remember that 
the sequence $x_1x_2\cdots x_k$ may contain the same element repeated a 
number of times.
For $k=0$ set
\[
\cS_0(P,\cF) = \cF(1),
\]
and for $k<0$, we have $\cS_k(P,\cF) = 0 $.
The differential $d_k\colon \cS_k(P,\cF) \ra \cS_{k-1}(P,\cF)$ is defined for $k>1$ by
\[
d_k(\lambda x_1x_2\cdots x_k)= \cF_{x_1}^{x_2}(\lambda)x_2\cdots x_k 
- \sum_{i=2}^k(-1)^i\lambda x_1 \cdots \widehat{x}_i \cdots x_k,
\]
and $d_1$ is defined by
\[
d_1(\lambda x) = \cF_{x}^{1}(\lambda).
\]

\begin{lemma}
$\cS_*(P,\cF)$ is a chain complex.
\end{lemma}

\begin{proof}
We need to show $d^2=0$. It is not hard to see that $d_{k-1}(d_k(\lambda x_1x_2 \cdots x_k))$ is a sum
of terms of the form $\mu x_1 \cdots \widehat{x}_i \cdots
\widehat{x}_j \cdots x_k$, where each indexing multi-sequence
appears exactly twice. We need to
check that such pairs have opposite signs. 
One such term arises by the deletion of $x_i$ and then $x_j$, with its
sign being
$(-(-1)^i) \times
(-(-1)^{j-1}) = (-1)^{i+j-1}$. On the other hand if $x_j$ is deleted
first then the sign is $(-(-1)^j)\times (-(-1)^i) = (-1)^{i+j}$, hence 
the pairs cancel.
\qed
\end{proof}

Given a morphism of coloured posets $(f,\tau)\colon (P_1,\cF_1) \ra (P_2,\cF_2)$ 
there is an induced map 
\[
f_*\colon \cS_*(P_1,\cF_1) \ra \cS_*(P_2,\cF_2)
\]
defined by
\[
\lambda x_1x_2\cdots x_k \mapsto \tau_{x_1}(\lambda)f(x_1)f(x_2) \cdots f(x_k).
\]

\begin{lemma}
$f_*$ is a well-defined chain map.
\end{lemma}

\begin{proof}
Clearly $\tau_{x_1}(\lambda)f(x_1)f(x_2) \cdots f(x_k)$ is an element of 
$\cS_*(P_2,\cF_2)$ and by the first condition for a morphism of coloured 
posets we also have $f(x_k) \neq 1_{P_2}$.
To see that $f_*$ is a chain map we calculate
\[
d(f_*(\lambda x_1x_2\cdots x_k)) = \cF_{f(x_1)}^{f(x_2)}(\tau_{x_1}(\lambda))
f(x_2)\cdots f(x_k)
+\Phi,
\]
and
\[
f_*(d(\lambda x_1x_2\cdots x_k)) =  \tau_{x_2}(\cF_{x_1}^{x_2}(\lambda))f(x_2)\cdots f(x_k)
+\Phi,
\]
for $\Phi=- \sum_{i=2}^{k}(-1)^i\tau_{x_1}(\lambda)f(x_1)\cdots 
\widehat{f(x_i)}\cdots f(x_k)$.
These are equal by the naturality of $\tau$.
\qed
\end{proof}

We have thus defined a covariant functor
\[
\cS_*\colon \cpr \ra \chr,
\]
from coloured posets to chain complexes over $R$. 
Finally, we define the {\em homology of the coloured poset } $(P,\cF)$ to be
\[
\hpc n = H_n (\spf *).
\]
Since homology is a functor from chain complexes to graded $R$-modules
we therefore have a covariant functor
\[
H_*\colon \cpr \ra \grrmod.
\]

Just as for homology of spaces we can cut down the size of the chain
complex by factoring out redundancies. We define $\cC_*(P,\cF)$
identically to $\cS_*(P,\cF)$, but with the additional requirement
that the $\x=x_1x_2\cdots x_k$ appearing in 
(\ref{complex:S}) are now {\em sequences}
$\x=x_1 <x_2 \cdots < x_k$. More precisely, for $k>0$ let
\[
\cC_k(P,\cF) = \bigoplus_{\substack{x_1x_2\cdots x_k\\x_i<x_{i+1}}}
\kern-2mm\cF(x_1),
\]
with the $x_i\in P\setminus 1$ as before.
Thus we have a direct summand for each length $k$ ordered sequence 
$x_1<x_2<\cdots <x_k$ in $P\setminus 1$. For $k=0$ we set
\[
\cC_0(P,\cF) = \cF(1),
\]
and $\cC_k(P,\cF) = 0$ for $k<0$. 
Clearly $\CC_k\subset\SS_k$, and we define the differential to be the
restriction to $\CC_k$ of the differential on $\SS_k$ (and so 
$\CC_*$ is a subcomplex of $\SS_*$).
Note that if there is a maximal length $r_0$ of an ordered sequence in $P$,
then
we have $\cC_k(P,\cF) = 0$ for $k> r_0$, which is not the
case for $\cS_*$. Nevertheless, it turns out that $\cC_*(P,\cF)$ is homotopy
equivalent to $\cS_*(P,\cF)$ as we will now see.

Let $\cD_k\subset\cS_k$ be the sub-module containing
those summands indexed by sequences with at least one repeat
ie: generated by elements of the form $\lambda x_1x_2\cdots x_k$ where
$x_i=x_{i+1}$ for at least one $i$. If $\lambda x_1\cdots xx\cdots x_k$ is one such,
then there are only two terms in $d(\lambda x_1\cdots xx\cdots x_k)$ without the
repeated $x$'s, and these have opposite signs, hence cancel
(if $\lambda xxx_3\cdots x_k$ is the term, then recall that $\FF_x^x=\id$).
Thus $d$ is closed on $\cD_*$ so $\cD_*\subset\cS_*$ is a subcomplex, and we have proved,

\begin{lemma}
There is a decomposition of complexes
\[
\cS_*(P,\cF) = \cC_*(P,\cF)\oplus \cD_*(P,\cF).
\]
\end{lemma}

But the complex $\cD_*$ proves not to be interesting, as the following proposition 
shows. The result is similar to standard results is algebraic topology, but it is
simple enough to write down an explicit proof, so we include it for completeness.

\begin{proposition} There is a homotopy equivalence of chain complexes
\[
\cD_*(P,\cF) \simeq 0:=
(\xymatrix{
\cdots \ar[r]^-{0}
& 0\ar[r]^-{0}
& 0\ar[r]^-{0} 
& \cdots})
\]
\end{proposition}

\begin{proof}
It suffices to show that the identity map on $\cD_*=\cD_*(P,\cF)$ is null homotopic. 
For this we need a family of maps $h_i\colon \cD_i \ra \cD_{i+1}$ such that
\begin{equation}\label{eq:hty}
\id = h_{i-1}d_i + d_{i+1}h_i.
\end{equation}
Given a multi-sequence $\x= x_1x_2\cdots x_k$ in $\cD_*$ we define
$p=p({\bf x})=\text{min}\{i\mid x_i=x_{i+1}\}$,
so $p({\bf x})$ is the position of the first repeating element,
and let $n=n({\bf x})$ be the number of times $x_p$ repeats. 
Thus we can write a multi-sequence $\x= x_1x_2\cdots x_k$ as
$ x_1 \cdots x_{p-1}x_p^n x_{p+n} \cdots x_k, $  
where $x_i\neq x_{i+1}$ for $1\leq i < p$ and $n\geq 2$.
Define $h_i\colon \cD_i \ra \cD_{i+1}$ by
\[
h_i(\lambda x_1 \cdots x_{p-1}x_p^n x_{p+n} \cdots x_k) = 
\begin{cases}
(-1)^{p+1 }\lambda x_1 \cdots x_{p-1}x_p^{n+1} x_{p+n} \cdots x_k & n \text{ even}\\
0 & n \text{ odd}
\end{cases}
\]
To show (\ref{eq:hty}) 
we consider separately the two cases $n$ odd and even.

\subsubsection*{The case: $n$ odd.} We have,
\begin{align*}
d_i(\lambda \bfx) = & \cF_{x_1}^{x_2}(\lambda)x_2  \cdots x_{p-1}x_p^n x_{p+n} \cdots x_k
      - \sum_{j=2}^{p-1} (-1)^j \lambda x_1 \cdots \widehat{x}_j \cdots x_{p-1}x_p^n x_{p+n} \cdots x_k\\
    &  - (-1)^p \lambda x_1 \cdots x_{p-1}x_p^{n-1} x_{p+n} \cdots x_k
      - \sum_{j=p+n}^{k} (-1)^j \lambda x_1 \cdots x_{p-1}x_p^n x_{p+n}  \cdots \widehat{x}_j \cdots x_k.
\end{align*}
Note how the $n$ terms indexed by $x_1\ldots x_{p-1}x_p^{n-1}x_{p+n}\ldots x_k$
cancel to give a single term when $n$ is odd.
Applying $h_{i-1}$ then gives zero on all terms except 
$- (-1)^p \lambda x_1 \cdots x_{p-1}x_p^{n-1} x_{p+n} \cdots x_k$, so that,
\begin{align*}
h_{i-1}d_i(\lambda \bfx) & =   h_{i-1}(- (-1)^p \lambda x_1 \cdots x_{p-1}x_p^{n-1} x_{p+n} \cdots x_k)\\
 & = - (-1)^p (-1)^{p+1}\lambda x_1 \cdots x_{p-1}x_p^{n} x_{p+n} \cdots x_k = \lambda \bfx,
\end{align*}
resulting in
$h_{i-1}d_i(\lambda \bfx) + d_{i+1}h_i(\lambda \bfx) 
= \lambda \bfx + d_{i+1}(0) = \lambda \bfx$.

\subsubsection*{The case: $n$ even.} We compute,
\begin{align*}
h_{i-1}d_i(\lambda \bfx)  =  (-1)^p \cF_{x_1}^{x_2}(\lambda) & x_2  \cdots x_{p-1}x_p^{n+1} x_{p+n} \cdots x_k \\
    &  - (-1)^p \sum_{j=2}^{p-1} (-1)^j \lambda x_1 \cdots \widehat{x}_j \cdots x_{p-1}x_p^{n+1} x_{p+n} \cdots x_k\\
    & - 0 
      - (-1)^{p+1}\sum_{j=p+n}^{k} (-1)^j \lambda x_1 \cdots x_{p-1}x_p^{n+1} x_{p+n}  \cdots \widehat{x}_j \cdots x_k.
\end{align*}
We also have
\begin{align*}
d_{i+1}h_{i}(\lambda \bfx) & =  d_{i+1}( (-1)^{p+1}\lambda  x_1 \cdots x_{p-1}x_p^{n+1} x_{p+n} \cdots x_k )\\
 &  = (-1)^{p+1} \cF_{x_1}^{x_2}(\lambda)  x_2  \cdots x_{p-1}x_p^{n+1} x_{p+n} \cdots x_k\\
 &   \;\;\;\;\;\;\;\;\;- (-1)^{p+1} \sum_{j=2}^{p-1} (-1)^j \lambda x_1 \cdots \widehat{x}_j \cdots x_{p-1}x_p^{n+1} x_{p+n} \cdots x_k\\
    &  \;\;\;\;\;\;\;\;\;- (-1)^{p+1}(-1)^p \lambda x_1 \cdots x_{p-1}x_p^{n} x_{p+n} \cdots x_k \\
    &  \;\;\;\;\;\;\;\;\;- (-1)^{p+1}\sum_{j=p+n}^{k} (-1)^{j+1} \lambda x_1 \cdots x_{p-1}x_p^{n+1} x_{p+n}  \cdots \widehat{x}_j \cdots x_k.
\end{align*}
Thus
\[
h_{i-1}d_i(\lambda \bfx) + d_{i+1}h_i(\lambda \bfx) = 
- (-1)^{p+1}(-1)^p \lambda x_1 \cdots x_{p-1}x_p^{n} x_{p+n} \cdots x_k = \lambda \bfx,
\]
as required.
\qed
\end{proof}

\begin{corollary}
There is a homotopy equivalence of chain complexes $\cpc * \simeq \spf *$.
\end{corollary}

In particular, $H_n(P,\FF)\cong H_n(\CC_*(P,\FF))$, a form more amenable to calculation.

We now briefly elaborate on the connection with traditional
(uncoloured) poset homology and in particular the assertion that we
have a local coefficient system on the order complex. An abstract
simplicial complex may be viewed as a category with objects $\Delta$
and a unique morphism $\ss\rightarrow\tau$ whenever $\tau\subset\ss$.
A system of local coefficients on $\Delta$ is a (covariant) functor
$\Delta\rightarrow\rmod$ (cf \cite{Gelfand99}*{\S 2.4}). One can form
the chain complex $\cB_*(\Delta,\FF)$ with $$
\cB_k=\bigoplus_{\ss\in\Delta_k}\FF(\ss),
$$
the direct sum over the $k$-simplicies. If $\ss=(x_0<\cdots<x_k)$ is one such
and $\ss_j=(x_0<\cdots<\wh{x}_j<\cdots<x_k)$, then the differential is
$$
d(\lambda\ss)=\sum_{j=0}^k (-1)^j\,\FF(\ss\rightarrow\ss_j)(\lambda)\ss_j.
$$

If $\FF$ is a constant system of local coefficients,
$\FF(\ss)=A$ for all $\ss\in\Delta$ and some $A\in\rmod$, and
$\FF(\ss\rightarrow\ss_j)=\id_A$, then this complex is the one appearing
in traditional poset topology: if $\Delta$ is the Folkman complex of $P$ (i.e. the 
order complex of $P\setminus 0,1$) 
then its homology is the {\em order homology of } $P$ with coefficients in $A$.

There is an augmented  version $\wtl{\cB}_*(\Delta,\FF)$ with $\wtl{\cB}_k=\cB_k$ for 
$k\geq 0$, and $\wtl{\cB}_{-1}=\FF(\varnothing)$. The extended differential
$d:\wtl{\cB}_0\rightarrow\wtl{\cB}_{-1}$ is given by the augmentation
$d(\lambda\ss)=\FF(\ss\rightarrow\varnothing)(\lambda)$.

Now if $P$ is a poset with $1$ and $\FF:P\rightarrow\rmod$ a
colouring, then we get a system of local coefficients on the order
complex $\FF_P:\Delta(P)\rightarrow\rmod$ given by
$\FF_P(\ss)=\FF(x_0)$ when $\ss=(x_0<\cdots<x_k)$, and
$\FF_P(\varnothing)=\FF(1)$. If $\ss\rightarrow\ss_j$ is a morphism in
$\Delta(P)$ where $\ss_j=(x_0<\cdots<\wh{x}_j<\cdots<x_k)$, then
$\FF_P(\ss\rightarrow\ss_j)=\FF_{x_0}^{x_1}$ when $j=0$, and is the
identity otherwise. We may restrict $\FF_P$ to a system of local
coefficients on the subcomplex $\Delta(P\setminus 1)$, and in doing so
we get the explicit connection we are looking for,

\begin{proposition}\label{section:homology:result600}
$\CC_*(P,\FF)=\wtl{\cB}_{*-1}(\Delta(P\setminus 1),\FF_P)$.
\end{proposition}

The proof is just of matter of unraveling the various definitions.
The advantage of this formulation is that to a certain extent it
allows us to appeal to the existing theory of lattice homology. For
example, if $P$ is a poset with $0$, then the order complex
$\Delta(P)$ is a cone on $\Delta(P\setminus 0)$, and so the (reduced)
order homology of an uncoloured poset with $0$ is trivial. Indeed, the
same happens in the coloured case when the colouring is constant:

\begin{example}\label{ex:homologytrivialcolouring}
Let $P$ be a poset with minimal element $0_P$ 
and $\FF$ the constant 
colouring by the $R$-module $A$. Then
$\cC_*(P,\cF)$ is acyclic i.e. $H_n(P,\cF) = 0$ for all $n$. 
This follows immediately from the above and Proposition \ref{section:homology:result600} 
together with a little care in degrees zero and one.
\end{example}

\begin{example}
If $(P_1,\cF_1)$ and $(P_2,\cF_2)$ are coloured posets then
there is a decomposition of complexes 
$$
\cS_*(P_1\cup P_2,
\cF_1\cup \cF_2) \cong \cS_*(P_1,\cF_1) \oplus \cS_*(P_2,\cF_2),
$$
inducing an isomorphism,
$$
\xymatrix{
H_*((P_1, \cF_1) \cup (P_2, \cF_2)) \ar[r]^-{\cong} 
& H_*(P_1, \cF_1) \oplus H_*(P_2, \cF_2).
}
$$
The essential point here is that elements of $P_1$ and elements of
$P_2$ are incomparable in $P_1 \cup P_2$ so an ordered sequence ${\bf
x}$ in $(P_1\cup P_2) \setminus 1$ is either completely in $P_1\setminus
1$ or completely in $P_2\setminus 1$. Moreover the differential respects this
splitting.
\end{example}

\subsubsection*{Cohomology.}
By defining 
\[
\cS^*(P,\cF) = \Hom_R(\cS_*(P,\cF), R)
\]
one can define {\em cohomology} $H^*(P,\cF)$ as the homology of the resulting 
cochain complex. If $R$ is a field then the universal coefficient 
theorem gives an isomorphism $H^*(P,\cF)\cong H_*(P,\cF)$.

\begin{example} Let $P$ be graded of rank $n$ with both a 0 and a 1.
Let 
$P^{\,\text{op}}$ be the opposite poset
defined by $x\leq y$ in $P^{\,\text{op}}$ if and only if $y\leq x$ in
$P$.
If we consider $P$ as a category, $P^{\,\text{op}}$ is simply the
opposite category. Since $P$ has a 0,  
$P^{\,\text{op}}$ has a 1. If we have a colouring
functor $\cF\colon P \ra
\rmod$ then by composing this with the functor $(-)^\vee\colon \rmod
\ra \rmod$ taking a module $A$ to its dual $\Hom_R(A,R)$, 
we get a contravariant functor $P \ra \rmod$. Equivalently, we have a
covariant functor $P^{\,\text{op}} \ra \rmod$, or in other words a colouring
$\cF^{\vee}$ of $P^{\,\text{op}}$. Explicitly, $\cF^{\vee}(x) =
\cF(x)^{\vee} = \Hom_R(\cF(x), R)$ and for $g\in \Hom_R(\cF(x), R)$
we have $\cF^{\vee}(x<y)(g) = g \circ \cF(y<x)$.

In this situation we have the following duality result
$$
H_k(P^{\,\text{op}},\cF^{\vee}) \cong H^{n-k}(P,\cF),
$$
seen by observing that
\begin{align*}
\cS^{n-k}(P, \cF) = \Hom_R( & \cS_{n-k}(P,\cF), R)=   
\Hom_R\biggl(\bigoplus_{\bf x}\cF(x_1), R\biggr)\\
& = 
\bigoplus_{\bf x} \Hom_R(\cF(x_1), R)= 
 \bigoplus_{\bf x} \cF(x_1)^{\vee} = \cS_k(P^{\,\text{op}}, \cF^{\vee}).
\end{align*}
\end{example}

\section{A long exact sequence for the poset obtained by gluing along a morphism}
\label{section:les}

In this section we show that a map $(f,\tau):(P_1,\FF_1)\ra (P_2,\FF_2)$
of coloured posets yields a long exact sequence in homology for the
coloured poset $P_1\bigcup_f P_2$ obtained
by gluing along the morphism.
The main spin-off occurs when 
we focus on Boolean lattices, where 
the decomposition of Example \ref{example:glueings} yields a long exact 
sequence in the homology of the three ingredients. This is the main technical 
tool needed to show
that for a Boolean lattice, the coloured poset homology defined in the last 
section
agrees with Khovanov's cube homology which we discuss in Section 4.

Given coloured posets $(P_1,\FF_1)$ and $(P_2,\FF_2)$ and a morphism
$(f,\tau):(P_1,\FF_1)\ra (P_2,\FF_2)$ we can form  the three complexes 
$\CC_*(P_i,\FF_i)$ for
$i=0,1$ and $\CC_*(P_1\bigcup_f P_2,\FF)$. It is clear that $\CC_*(P_2,\FF_2)$ 
is a sub-module of 
$\CC_*(P_1\bigcup_f P_2,\FF)$, but also, one easily checks that 
$d(\CC_*(P_2,\FF_2))\subset\CC_*(P_2,\FF_2)$, and so there is 
a short exact sequence of complexes
$$
\xymatrix{
0 \ar[r]
&  \cC_{*}(P_2,\FF_2) \ar[r]^-{i}
& \CC_*(P_1\bigcup_f P_2,\FF)\ar[r]^-{q} 
&  Q_*  \ar[r]   & 0},
$$
where by definition, $Q_*$ is the quotient. This yields a long exact sequence in 
homology,
\begin{equation}\label{eq:lesQ}
\xymatrix{
\cdots \ar[r]^-{\delta}
& H_n(P_2,\FF_2) \ar[r]^-{i_*}
& H_n (P_1\bigcup_f P_2,\FF) \ar[r]^-{q_*}
& H_{n}(Q_*) \ar[r]^-{\delta}
& H_{n-1}(P_2,\FF_2) \ar[r]^-{i_*}  & \cdots}
\end{equation}

The $n$-chain module $Q_n$ of the quotient complex is isomorphic to
$$
\bigoplus_{\x}\FF(x_1),
$$
the direct sum over those sequences $\x$ in $P$ not entirely contained in $P_2$,
i.e. $\x=x_1\ldots x_n$ with $x_i\in P_1$ or $\x=x_1\ldots x_jy_1\ldots y_{n-j}$
where $0<j<n$ and the $x_i\in P_1$,  $y_i\in P_2\setminus 1$. We will
write $\x=x_1\ldots x_jy_1\ldots y_{n-j}$ for the generic sequence,, with the
understanding that $\x=x_1\ldots x_n$ when $j=n$.
The differential is given by $d(\lambda\x)=\aa_j+\bb_j$, where $\aa_1=0$ and
$$
\aa_j= \FF_{x_1}^{x_2}(\lambda)x_2\ldots x_jy_1\ldots y_{n-j}
+\sum_{k=2}^j (-1)^{k-1}\lambda x_1\ldots\wh{x}_k\ldots x_jy_1\ldots y_{n-j},
$$
for $j>1$, and $\bb_n=0$, and 
$$
\bb_j=\sum_{k=1}^{n-j}(-1)^{j+k-1}
\lambda x_1\ldots x_jy_1\ldots\widehat{y}_k\ldots y_{n-j},
$$
for $j<n$.

We now define  $\pi_n:Q_n \ra \CC_{n-1}(P_1,\FF_1)$
by
\begin{equation}\label{eq:pi}
\pi(\lambda\,x_1\ldots x_j y_1\ldots y_{n-j})
=\left\{\begin{array}{l}
\lambda\,x_1\ldots x_{n-1},\text{ if }j=n\text{ and }x_n=1_{P_1},\\
0,\text{ otherwise.}
\end{array}\right.
\end{equation}

It is routine to check that,
  
\begin{lemma}
$\pi:Q_* \ra \CC_{*-1}(P_1,\FF_1)$ is a chain map.
\end{lemma}

In fact while $Q_*$ is  \emph{a priori\/} a great deal bigger than
$\CC_*(P_1,\FF_1)$, it turns out to contain a number of acyclic subcomplexes,
allowing us to establish an isomorphism 
$H_{n}(Q_*)\cong H_{n-1}(P_1,\FF_1)$.

\begin{proposition}\label{prop:pi}
The induced map $\pi_*:H_{n}(Q_*)\rightarrow H_{n-1}(P_1,\FF_1)$ is an isomorphism.
\end{proposition}

Delaying the proof of this momentarily, we now combine this proposition with the
long exact sequence (\ref{eq:lesQ}) to obtain the  main result of this section.

\begin{theorem}\label{section:les:result100}
Let $(P_i,\FF_i)$, $i=1,2$, be coloured posets and 
$(f,\tau):(P_1,\FF_1)\ra(P_2,\FF_2)$ a morphism of coloured posets. Then there
is a long exact sequence,
$$
\xymatrix{
\cdots \ar[r]
& H_n(P_2,\FF_2) \ar[r]^-{i_*}
& H_n (P_1\bigcup_f P_2,\FF) \ar[r]^-{(\pi q)_*}  
& H_{n-1}(P_1,\FF_1) \ar[r]^-{\delta\pi_*^{-1}}
& H_{n-1}(P_2,\FF_2) \ar[r]  & \cdots}
$$
where $(P_1\bigcup_f P_2,\FF)$ is the coloured poset obtained
by gluing along the morphism.
\end{theorem}

Before proving Proposition \ref{prop:pi} we 
introduce a number of auxillary complexes that play a role in the analysis
of the homology of $Q_*$.  For $p>0$, fix $\x=x_1\ldots x_p$ in $P_1$, 
and define a complex $A^\x_*$ by setting
$$
A^\x_q=\bigoplus_{\x\,y_1\ldots y_q}\kern-2.5mm\FF(x_1),
$$
ie: the direct sum over those $x_1\ldots x_py_1\ldots y_q$ where the $x$'s
are fixed and the $y$'s in $P_2$ are allowed to vary. 
Let $A^\x_0=\FF(x_1)$.
The differential is given by 
$$
d(\lambda\x\,y_1\ldots y_q)=\sum_{k=1}^{q}(-1)^{p+k-1}
\lambda\x\,y_1\ldots\wh{y}_k\ldots y_{q},
$$
and $d(\lambda\x\,y)\mapsto\lambda\in A^\x_0$.

Let $x\in P_1$ and $P^x$ those $y\in P_2$ with $x<y$ in $P$. Then is is easy
to see that $P^x$ is a subposet of $P$ with unique minimal element $f(x)$.
Let $\ve_p=1$ when $p$ is even and $\ve_p=(-1)^{p+q}$ when $p$ is odd.

\begin{lemma}\label{section:les:result200}
The map $\lambda\x y_1\ldots y_q\mapsto\ve_p\lambda y_1\ldots y_q$ is an
isomorphism of complexes
$$
\xymatrix{A^\x_* \ar[r]^-{\cong} & \CC_*(P^{x_p},\FF_\x)},
$$
where $\x=x_1\ldots x_p$ and
$\FF_\x$ is the constant colouring $\FF_\x(x)=\FF(x_1)$. In particular,
$A^\x_*$ is acyclic.
\end{lemma}

The proof is elementary, and no doubt the reader can provide the details by
scrutinizing the picture,
$$
\begin{pspicture}(0,0)(15,3.5)
\rput(0,-0.2){
\rput(7.5,2){\BoxedEPSF{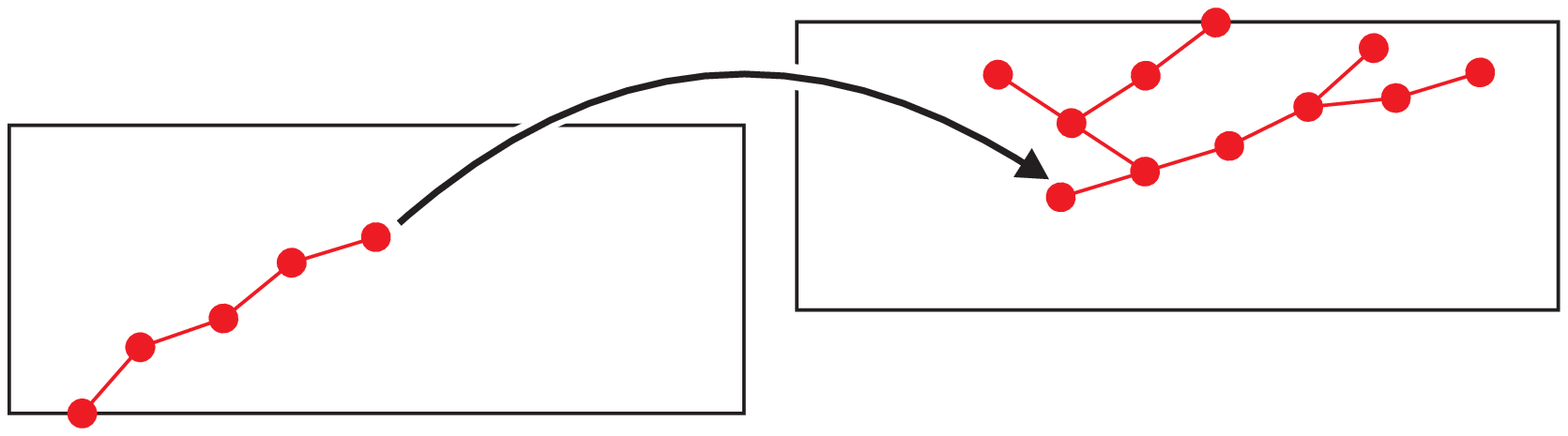 scaled 800}}
\rput(1.3,.65){$x_1$}\rput(4,1.4){$x_p$}
\rput(6.8,3.5){$f$}
\rput(6,1){$P_1$}\rput(13,2){$P_2$}
\rput(9.8,1.75){$f(x_p)$}
}
\end{pspicture}
$$
while keeping a close eye on the signage. 
Acyclic-ness follows from Example \ref{ex:homologytrivialcolouring}.

\begin{proof}[of Proposition \ref{prop:pi}] 
Let $A_n\subset Q_n$ be the direct sum
$\oplus_\x\FF(x_1)$ over those $\x=x_1\ldots x_{n-1}x_n$ where $x_n=1_1:=1_{P_1}$,
the unique maximal element in $P_1$, and decompose (as modules),
$$
Q_n = A_n \oplus D_n.
$$
Thus, $D_n$ is the direct sum over those $\x$ in $Q_n$ not finishing at $1_1$. It is
readily checked that 
$d(D_n)\subset D_{n-1}$, giving that 
$(D_*, d)$ is a subcomplex of $(Q_*,d)$. 

We now show that $D_*$ is acyclic i.e. $H_n(D_*) = 0$ for all $n$. We show this
by filtering $D_*$ and analyzing the associated spectral sequence. 
Let $p>0$, and set
\[
F_pD_n=\bigoplus_{\x}\FF(x_1),
\]
 the direct sum over those 
$\x=x_1\ldots x_jy_1\ldots y_{n-j}$ where now $1\leq j\leq p$ (and as usual,
the $x_i\in P_1$ and the $y_i\in P_2\setminus 1$). Set $F_pD_*=0$ for $p<1$.
Thus, $F_pD_n$ consists of those modules indexed by sequences of length $n$ which
have exited $P_1\subset P$ after the $p$-th element.

It is easy to check that $F_pD_*$ is a sub-complex of $D_*$, yielding
for each $n$ a bounded  filtration,
$$
0=F_0D_n\subset F_1D_n
\subset\cdots\subset F_{p-1}D_n\subset F_pD_n\subset F_{p+1}D_n\subset\cdots
\subset F_nD_n=D_n.
$$
This gives rise to a first quadrant spectral sequence converging  to $H_*D$.
The $E^0$-page is 
$$
E^0_{p,q}=
\frac{F_pD_{p+q}}{F_{p-1}D_{p+q}}=
\kern-3mm\bigoplus_{x_1\ldots x_py_1\ldots y_{q}}
\kern-5mm\FF(x_1),
$$
with differential $d^0\colon E^0_{p,q} \ra E^0_{p,q-1}$ defined by
$$
d^0(\lambda x_1\ldots x_py_1\ldots y_{q}) = 
\sum_{k=1}^q (-1)^{p+k-1}\lambda x_1\ldots x_p y_1 \ldots \widehat{y}_k \ldots y_q.
$$
Note that $E^0_{0,q}=0$ for all $q$, and thus $E^\infty_{0,q}=0$.
To compute the $E^1$-page we fix $p$ and consider separately the cases $q>0$
and $q=0$.

\subsubsection*{The case $q>0$.}
We show that any cycle in $E^0_{p,q}$ is also a boundary, and thus $E^1_{p,q}=0$.
Let 
$$
\ss=\sum_i \lambda_i x_{i,1}\ldots x_{i,p}y_{i,1}\ldots y_{i,q}
$$
be a general element of $E^0_{p,q}$. Let $\x=x_1\ldots x_p$ be a \emph{fixed\/}
sequence in $P_1$ and 
$$
\ss^\x=\sum_j \lambda\x\,y_{j,1}\ldots y_{j,q}
$$
the sum of those terms in $\ss$ with $x_{i,1}\ldots x_{i,p}=\x$. Then
$\ss=\sum_\x\ss^\x$, the sum over those $\x$ appearing as initial
segments in $\ss$, and
$$
d^0\ss=\sum_\x d^0\ss^\x.
$$
Thus, if $A_*^\x$ is the complex defined
immediately prior to Lemma \ref{section:les:result200}, 
now a subcomplex of $E^0_{p,*}$,
then $\ss^\x\in A_*^\x\subset E^0_{p,*}$, and
$d^0\ss^\x\in A_{*-1}^\x\subset E^0_{p,*-1}$. Also, if $\x\not=\w$ then
$A_*^\x\cap A_*^\w=\{0\}$ as subcomplexes of $E^0_{p,*}$, and so
$\ss$ is a cycle if and only if each $\ss^\x$ is a cycle.
But the $A_*^\x$ are acyclic, so $\ss^\x=d^0\tau^\x$ for some
$\tau^\x\in A_{*+1}^\x \subset E^0_{p,*+1}$, giving $\ss=d^0(\sum\tau^\x)$,
and thus $E^1_{p,q}=0$ as claimed.

\subsubsection*{The case $q=0$.}
Here $d^0=0$ and so the cycles are all of 
$E^0_{p,0}=\bigoplus_{\x}\FF(x_1)$, where
$\x=x_1\ldots x_p$ with $x_p\not=1_1$.
We show that $d^0:E^0_{p,1} \ra E^0_{p,0}$ is onto and
conclude that $E^1_{p,0} = 0$. 
If $\lambda x_1\ldots x_p$ is an element with $x_p\neq
1_1$, then $f(x_p)\in P_2$ is $\not=1_2=1$,
and $x_p<f(x_p)$. We then have $d^0(\lambda
x_1\ldots x_pf(x_p))=\lambda x_1\ldots x_p$ as required.

\subsubsection*{}
Thus the $E^1$-page of the spectral sequence is entirely trivial,
so that in the induced filtration of $H_*D$,
$$
\cdots \subset F_{p-1}H_n(D_*)\subset F_pH_n(D_*)\subset F_{p+1}H_n(D_*)\subset\cdots\subset 
H_n(D_*),
$$
we have trivial quotients. Thus  $F_{p-1}H_n=F_pH_n$ for all $p$ and $n$. As $F_0H_nD=0$, we
conclude that $H_n(D_*)=0$ as claimed.

To finish the proof observe that there is a short exact sequence
$$
\xymatrix{0 \ar[r]&  D_{*} \ar[r]& Q_*  \ar[r]^-{\pi} &  A_* \ar[r]  & 0},
$$
whose associated homology long exact sequence, together with the acyclic-ness 
of $D_*$, gives that the quotient map
$\pi\colon Q_* \ra A_*$ induces 
isomorphisms $\pi_*:H_n(Q_*) \ra H_n(A_*)$. 
Now, $A_n=\bigoplus_{x_1\ldots x_n}\FF(x_1)$ with $x_n=1_1$, and thus
the complex $A_*$ can be identified with $\CC_{*-1}(P_1,\FF_1)$.
Under this identification the map $\pi$ above is the map
$\pi:Q_* \ra \CC_{*-1}(P_1,\FF_1)$ of complexes defined in (\ref{eq:pi}),
finishing the proof.
\qed
\end{proof}

Let $(\B,\FF)$ be a coloured Boolean lattice of rank $r$ 
and
$\B=\B_0\bigcup_f\B_1$ a decomposition of the form given in Example 
\ref{example:glueings}.

\begin{corollary}\label{corollary:les}
There is a long exact sequence
$$
\xymatrix{
\cdots \ar[r]
& H_n({\bB_1}, {\FF_1}) \ar[r]^-{i_*}  
& H_n (\bB, \FF) \ar[r]^-{(\pi q)_*}  
& H_{n-1}({\bB_0}, {\FF_0}) \ar[r]^-{\delta\pi_*^{-1}} 
& H_{n-1}( {\bB_1}, {\FF_1}) \ar[r]  
& \cdots
}
$$
\end{corollary}

\section{The cube complex of a Boolean lattice and its homology}
\label{section:cubecomplex}

We now recall a construction, first due to Khovanov \cite{Khovanov00}, of a complex
from a coloured Boolean lattice. It is central to the definition 
of the Khovanov homology of a link and is used in one of the recent 
combinatorial formulations of Heegaard-Floer knot homology. The reader should
be aware that we are grading everything {\em homologically}, whereas in the
applications cited above it is traditional to use cohomological conventions.

Let $\B$ be a Boolean lattice of rank $r$ with ordered
atoms $a_1,\ldots,a_r$, and colouring $\FF:\B\rightarrow\rmod$,
and recall the unique expression 
(\ref{joinofatoms})
for an element of $\B$ as a join of the $a_i$
(this replaces the conventions in earlier, non-lattice oriented, 
literature on Khovanov homology, where the elements of $\B$ were $r$-strings of
$0$'s and $1$'s, and the atoms those $r$-strings containing a single $1$).
Write $1:=1_\B$, the join of all the atoms.

If $x<_c y$, then let $\ve(x<_c y)=(-1)^j$ where $j$
is the number of atoms appearing before $a_\ell$ in the unique expression
for $y$ (see (\ref{joinofatoms}) and the comments following it).
If 
$\x=x_1<_c x_2<_c\cdots<_c x_k$ is a saturated sequence in $\B$, let
$$
\ve_\x=\ve(x_1\ldots x_k)
=\ve(x_1<_c x_2<_c\cdots<_c x_k):=\prod\ve(x_i<_c x_{i+1}).
$$
If $1_0=a_2\vee\cdots\vee a_r$, then observe that $\ve(1_0<_c 1)=1$.

Khovanov's \emph{cube complex \/}
$\KK_{*}(\B,\FF)$ is then defined to have chain
modules,
$$
\KK_k=\bigoplus_{\rk x=r-k}\kern-3mm\FF(x),
$$
and differential $d_k:\KK_k(\B,\FF)\rightarrow\KK_{k-1}(\B,\FF)$,
$$
d_k(\lambda)=\sum\ve(x<_c y)\FF_x^y(\lambda),
$$
where $\lambda\in \FF(x)$ with $\rk x=r-k$, and the sum is over all $y$ covering $x$.
Thus, $d(\KK_k)\subset\KK_{k-1}$ with $d=\sum_{\rk x=r-k}\ve(x<_c y)\FF_x^y$.
Observe that in degree zero the chains are just $\FF(1)$, in degree
$r$ they are $\FF(0)$, and $\KK_k=0$ outside of the range $0\leq k\leq r$.
To see that $d$ is a differential, observe that if $x<_c z<_c y$ in
$\B$, then there is a unique $z'$ with $x<_c z'<_c y$, and that
$\ve(x<_c z<_c y)=\ve(x<_c z'<_c y)$, ie: consecutive edges
of the Hasse diagram for a Boolean lattice can always be completed to form a 
square in a unique way, and all squares anticommute. As $d$ is a sum over such
squares we get $d^2=0$. 

Write $H^\diamond_*(\B,\FF)=H_*(\KK_*(\B,\FF))$ for the homology 
of the cube complex. It should be noted that $H^\diamond(-)$ is not
natural with respect to morphisms of coloured Boolean lattices in
general. It is, however, natural with respect to morphisms $(f,\tau)$ for
which $f$ is a co-rank preserving injection.

The Khovanov homology of an oriented link diagram is defined as (a normalised version of) 
the homology
of the cube complex associated to the coloured Boolean lattice defined in Example 
\ref{ex:khovanov}. A small class of (graded) Frobenius algebras result in a homology 
theory that is invariant under Reidemeister moves of diagrams, thus
giving a genuine invariant of links. The reader wishing to make this
precise should be warned that here our homological grading
conventions conflict with Khovanov's cohomological ones, and so care is
needed (see \S \ref{sec:normalisation}). 

Similarly, the combinatorial interpretation of Heegaard-Floer knot homology is defined as
the homology
of the cube complex associated to the coloured Boolean lattice of Example \ref{ex:os}.

\begin{figure}
\begin{pspicture}(0,0)(15,7)
\rput(0,-2){
\psline[linewidth=.3mm,linestyle=dotted,dotsep=2pt]{->}(5.7,3.5)(11.5,3.5)
\rput*(9,3.5){$\sum\ve(0<_c a_{i})\FF_{0}^{a_{i}}$}}
\rput(0,0){
\psline[linewidth=.3mm,linestyle=dotted,dotsep=2pt]{->}(6.7,3.5)(11.5,3.5)
\rput*(9,3.5){$\sum\ve(a_i<_c a_{ij})\FF_{a_i}^{a_{ij}}$}}
\rput(0,2){
\psline[linewidth=.3mm,linestyle=dotted,dotsep=2pt]{->}(5.7,3.5)(11.5,3.5)
\rput*(9,3.5){$\sum\ve(a_{ij}<_c 1)\FF_{a_{ij}}^{1}$}}
\psline[linewidth=.3mm,linestyle=dotted,dotsep=2pt]{->}(4.7,.5)(11.5,.5)
\psline[linewidth=.3mm,linestyle=dotted,dotsep=2pt]{->}(2.7,4.5)(11.2,4.5)
\psline[linewidth=.3mm,linestyle=dotted,dotsep=2pt]{->}(2.7,2.5)(11.2,2.5)
\psline[linewidth=.3mm,linestyle=dotted,dotsep=2pt]{->}(4.7,6.5)(11.5,6.5)
\rput(2,.2){
\rput(2,0){\psframe[fillstyle=solid,fillcolor=white](0,0)(.7,.6)\rput(.35,.3){$0$}}
\rput(0,2){\psframe[fillstyle=solid,fillcolor=white](0,0)(.7,.6)\rput(.35,.3){$a_{1}$}}
\rput(2,2){\psframe[fillstyle=solid,fillcolor=white](0,0)(.7,.6)\rput(.35,.3){$a_{2}$}}
\rput(4,2){\psframe[fillstyle=solid,fillcolor=white](0,0)(.7,.6)\rput(.35,.3){$a_{3}$}}
\rput(0,4){\psframe[fillstyle=solid,fillcolor=white](0,0)(.7,.6)\rput(.35,.3){$a_{12}$}}
\rput(2,4){\psframe[fillstyle=solid,fillcolor=white](0,0)(.7,.6)\rput(.35,.3){$a_{13}$}}
\rput(4,4){\psframe[fillstyle=solid,fillcolor=white](0,0)(.7,.6)\rput(.35,.3){$a_{23}$}}
\rput(2,6){\psframe[fillstyle=solid,fillcolor=white](0,0)(.7,.6)\rput(.35,.3){$1$}}
\psline[linewidth=.2mm]{->}(2,.7)(.8,1.9)
\psline[linewidth=.2mm]{->}(2.35,.7)(2.35,1.9)
\psline[linewidth=.2mm]{->}(2.7,.7)(3.9,1.9)
\psline[linewidth=.2mm]{->}(.35,2.7)(.35,3.9)
\psline[linewidth=.2mm]{->}(.7,2.7)(1.9,3.9)
\rput{45}(1.35,3.2){\psframe[fillstyle=solid,linecolor=white,
fillcolor=white](0,0)(.2,.2)}
\psline[linewidth=.2mm]{->}(2,2.7)(.8,3.9)
\psline[linewidth=.2mm]{->}(4,2.7)(2.8,3.9)
\rput{45}(3.35,3.2){\psframe[fillstyle=solid,linecolor=white,
fillcolor=white](0,0)(.2,.2)}
\psline[linewidth=.2mm]{->}(2.7,2.7)(3.9,3.9)
\psline[linewidth=.2mm]{->}(4.35,2.7)(4.35,3.9)
\psline[linewidth=.2mm]{->}(.7,4.7)(1.9,5.9)
\psline[linewidth=.2mm]{->}(2.35,4.7)(2.35,5.9)
\psline[linewidth=.2mm]{->}(4,4.7)(2.8,5.9)
\rput(1.2,1.3){${\scriptstyle\red{+}}$}
\rput(2.5,1.3){${\scriptstyle\red{+}}$}
\rput(3.5,1.3){${\scriptstyle\red{+}}$}
\rput(.2,3.3){${\scriptstyle\red{-}}$}
\rput(1.2,3){${\scriptstyle\red{-}}$}
\rput(1.9,3){${\scriptstyle\red{+}}$}
\rput(2.8,3){${\scriptstyle\red{-}}$}
\rput(3.5,3){${\scriptstyle\red{+}}$}
\rput(4.5,3.3){${\scriptstyle\red{+}}$}
\rput(1.1,5.3){${\scriptstyle\red{+}}$}
\rput(2.55,5.3){${\scriptstyle\red{-}}$}
\rput(3.6,5.3){${\scriptstyle\red{+}}$}
}
\rput(12,0.2){
\rput(0,0.3){$\FF(0)$}
\rput(0,2.3){$\bigoplus\FF(a_i)$}
\rput(0,4.3){$\bigoplus\FF(a_{ij})$}
\rput(0,6.3){$\FF(1)$}
\psline[linewidth=.2mm]{->}(0,.5)(0,2.1)
\psline[linewidth=.2mm]{->}(0,2.5)(0,4.1)
\psline[linewidth=.2mm]{->}(0,4.5)(0,6.1)
\rput(-.2,1.3){$d_3$}\rput(-.2,3.3){$d_2$}\rput(-.2,5.3){$d_1$}
}
\rput(1,.2){
\rput(0,0.3){rank $0$}
\rput(0,2.3){rank $1$}
\rput(0,4.3){rank $2$}
\rput(0,6.3){rank $3$}
}
\rput(14,.2){
\rput(0,0.3){degree $3$}
\rput(0,2.3){degree $2$}
\rput(0,4.3){degree $1$}
\rput(0,6.3){degree $0$}
}
\rput*(3.35,2.5){$\oplus$}\rput*(5.35,2.5){$\oplus$}
\rput*(3.35,4.5){$\oplus$}\rput*(5.35,4.5){$\oplus$}
\end{pspicture}
\caption{The cube complex $\KK_*(\B,\FF)$ for the Boolean lattice of rank $3$ (after
Bar-Natan \cite{Bar-Natan02}). The join $a_i\vee a_j$ has been abbreviated 
$a_{ij}$. The edges $x<_c y$ of the Hasse diagram for $\B$ have been labelled
with the Khovanov signage $\ve(x<_c y)$.}
\end{figure}
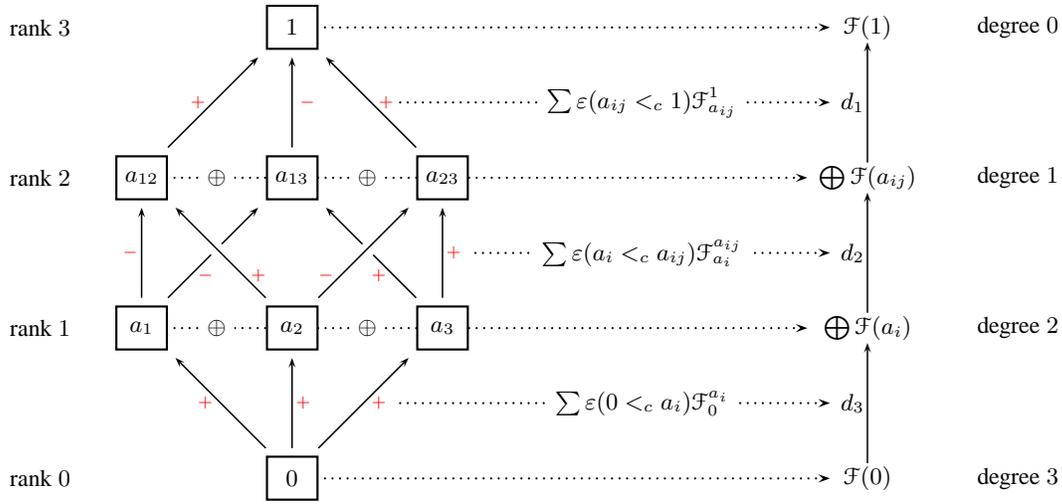

The decomposition $\B=\B_0\bigcup_f\B_1$ of Example \ref{example:glueings}
yields a long exact sequence similar to that obtained for coloured poset
homology described in the last section.
As the $\B_i$ $(i=0,1)$ are Boolean
of rank $r-1$, we may form the associated cube complexes
$\KK_*(\B_i,\FF_i)$ where $\FF_i$ is the restriction of $\FF$ to
$\B_i$. As with the complex $\CC_*$ in \S\ref{section:les}, 
$\KK_*(\B_1,\FF_1)$ is a subcomplex
of $\KK_*(\B,\FF)$, but now the quotient is considerably simpler, for
the map,
$$
\sum_{\rk x=r-k}\lambda_x=\sum_{x\not\in\B_0}\mu_x^{}+\sum_{x\in\B_0}\nu_x^{}
\mapsto\sum_{x\in\B_0}\nu_x^{},
$$
gives an isomorphism of complexes,
$\KK_*(\B,\FF)/\KK_*(\B_1,\FF_1)\rightarrow\KK_{*-1}(\B_0,\FF_0)$,
and thus a short exact sequence,
$$
0 \ra \KK_{*}(\B_1, \FF_1) \ra \KK_*(\B,\FF) \ra \KK_{*-1}(\B_0, \FF_0) \ra 0. 
$$
This sequence is well known, although the degree drop in our version happens
in the quotient, rather than the subcomplex, as we are grading $\KK_*$
homologically, rather than cohomologically.
Finally, we have the induced long exact sequence in homology,
$$
\xymatrix{
\cdots \ar[r]
& H_n^\diamond(\B_1,\FF_1) \ar[r] & H_n^\diamond(\B,\FF) \ar[r] &
H_{n-1}^\diamond(\B_0,\FF_0) \ar[r] & H_{n-1}^\diamond(\B_1,\FF_1)
\ar[r] & \cdots } 
$$ 

In Khovanov homology for links, if $(\B,\FF)$ is
the coloured Boolean lattice of a diagram $D$ (see Example
\ref{ex:khovanov}) then $(\B_0,\FF_0)$ and $(\B_1,\FF_1)$ can be interpreted as the 
coloured
lattices associated to diagrams $D_0$ and $D_1$ obtained from $D$ by resolving a chosen
crossing in $D$ to a 0- and 1-smoothing respectively. In this case the above 
long exact sequence
is a homological incarnation of the kind of skein relation found in the 
definition of certain knot polynomials.

\section{A quasi-isomorphism}\label{main:theorem}

We now have two chain complexes, and their homologies, associated to a coloured
Boolean lattice: the coloured poset  homology 
$H_*(\B,\FF)$ of the complex $\CC_*(\B,\FF)$ from
\S\ref{section:homology}, and the 
homology $H^\diamond(\B,\FF)$ of the cube complex defined in 
\S\ref{section:cubecomplex}.
In this section we describe a chain map $\phi$ from the cube complex
to $\CC_*(\B,\FF)$, and show that it turns out to be a quasi-isomorphism. The main
result is the following, whose proof appears at the end of the section.

\begin{theorem}\label{thm:main}
Let $(\B,\FF)$ be a coloured Boolean lattice. 
Then $\phi:\KK_*(\B,\FF)\ra\CC_*(\B,\FF)$ defined below
is a quasi-isomorphism, yielding isomorphisms,
$$
\xymatrix{H_{n}^\diamond(\B,\FF) \ar[r]^-{\cong} & H_{n}(\B,\FF)}.
$$
\end{theorem}

We now define the map $\phi$. 
Let $\lambda\in\FF(x)$ for $x\in\B$, and $\x=x_1<_c\cdots<_c x_k$ a saturated sequence
in $\B$ from $x$ to $1$, ie: with $x_1=x$ and $x_k=1$,
and let $\x^\circ=x_1<_c\cdots<_c x_{k-1}$.
Recalling the definition
of $\ve_\x\in\{\pm 1\}$ from \S\ref{section:cubecomplex}, set
$\phi:\KK_n(\B,\FF)\rightarrow\CC_n(\B,\FF)$ to be
$$
\phi(\lambda)=\sum_\x\ve_\x \lambda\,\x^\circ,
$$
the sum over all saturated sequences $\x\in\B$ from $x$ to $1$.

\begin{figure}
\begin{pspicture}(0,0)(15,4)
\rput(-2,0){
\rput(4,0.2){$\red{x=0}$}
\rput(2.7,1.4){$x_1$}\rput(4.3,1.4){$x_2$}
\rput(5.3,1.4){$x_3$}\rput(2.6,2.65){$x_{12}$}
\rput(4.35,2.65){$x_{13}$}\rput(5.4,2.65){$x_{23}$}
\rput(4,3.8){$1$}
\rput(3.2,1){${\red{\scriptstyle +}}$}\rput(4.15,1){${\red{\scriptstyle +}}$}
\rput(4.8,1){${\red{\scriptstyle +}}$}\rput(2.75,2){${\red{\scriptstyle -}}$}
\rput(5.25,2){${\red{\scriptstyle +}}$}
\rput(3.55,1.75){${\red{\scriptstyle +}}$}\rput(4.45,1.75){${\red{\scriptstyle -}}$}
\rput(3.55,2.35){${\red{\scriptstyle -}}$}\rput(4.45,2.35){${\red{\scriptstyle +}}$}
\rput(3.2,3.1){${\red{\scriptstyle +}}$}\rput(4.15,3.1){${\red{\scriptstyle -}}$}
\rput(4.8,3.1){${\red{\scriptstyle +}}$}
\rput(4,2){\BoxedEPSF{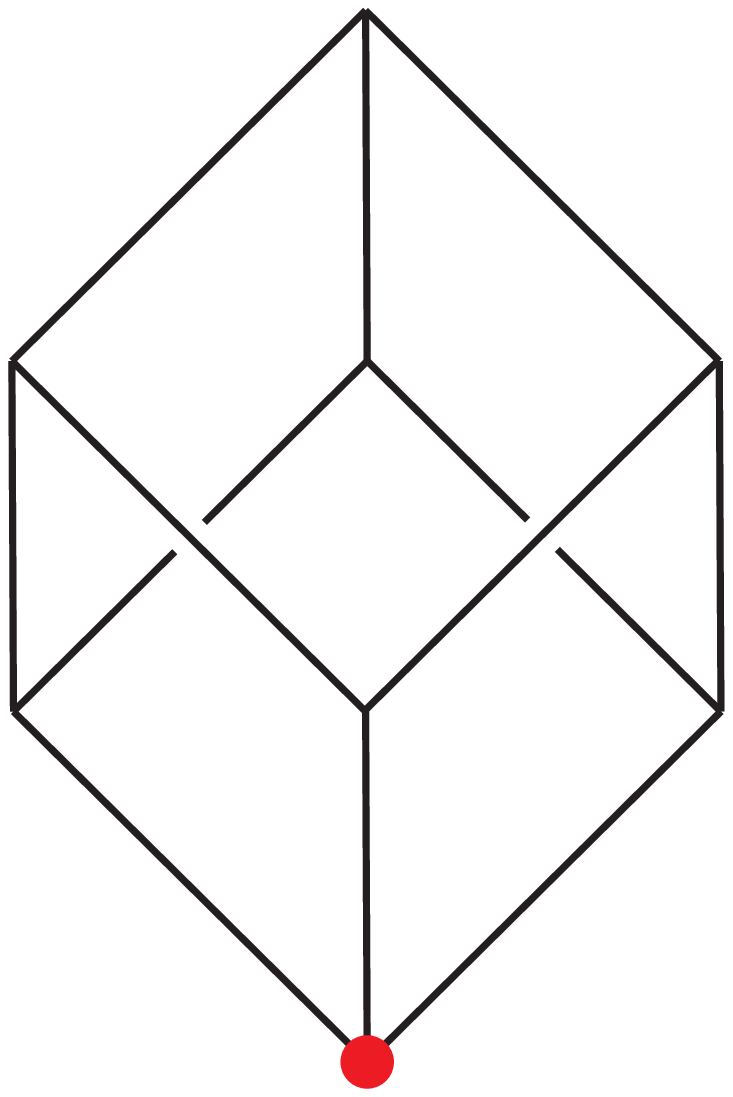 scaled 300}}
}
\psline[linewidth=.2mm]{->}(3.7,2)(4.6,2)
\rput(-1.5,0){
\rput(7.7,1.3){${\scriptstyle +}$}\rput(8.3,1.3){${\scriptstyle -}$}
\rput(6.4,2){${\scriptstyle -}$}\rput(9.6,2){${\scriptstyle +}$}
\rput(7.7,2.7){${\scriptstyle +}$}\rput(8.3,2.7){${\scriptstyle -}$}
\rput(8,2){\BoxedEPSF{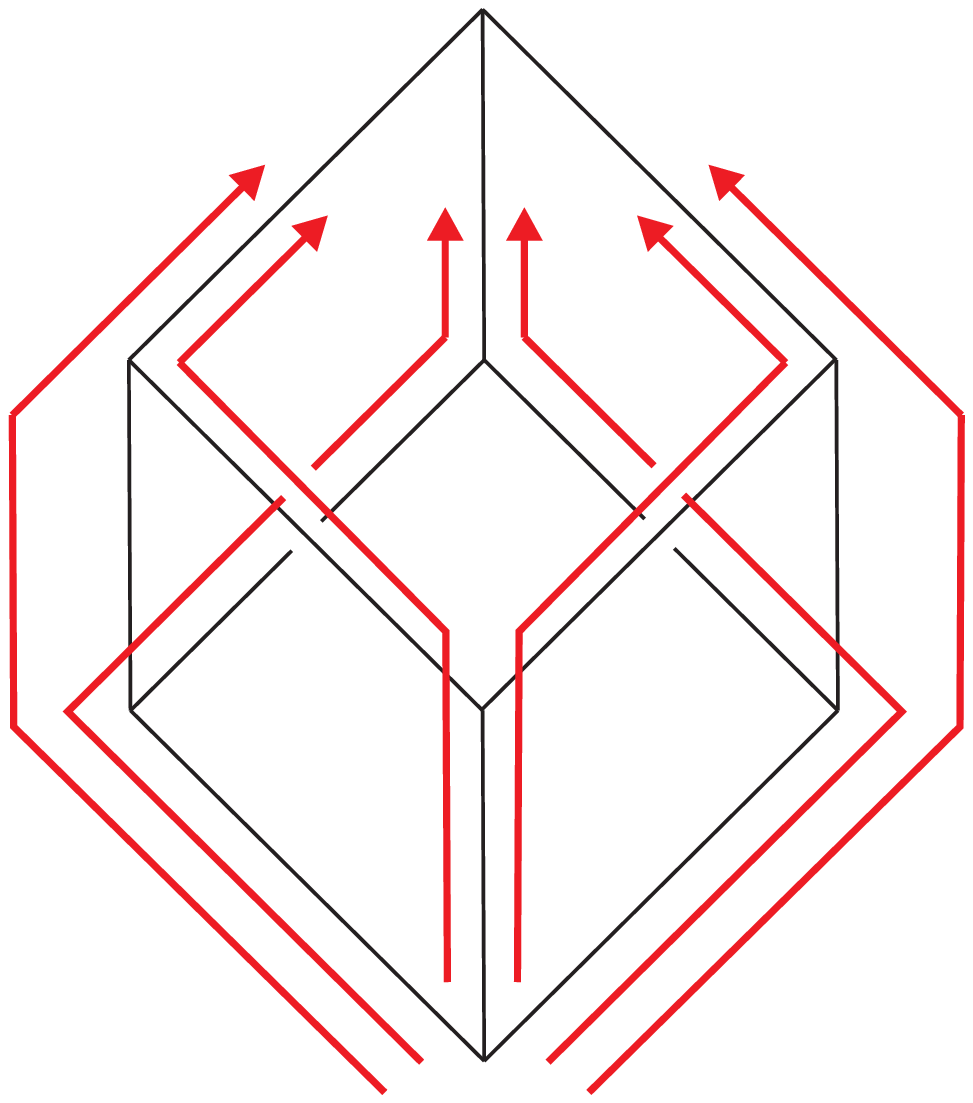 scaled 300}}
}
\psline[linewidth=.2mm]{->}(8.4,2)(9.3,2)
\rput(12.2,2){$\lambda\in\FF(x)
\kern-1mm\xymatrix{\ar@{|->}[r]^-{\phi}&}
\kern-2mm\left\{\begin{array}{l}-\lambda\,x_1x_{12}+\lambda\,x_1x_{13}\\
                       +\lambda\,x_2x_{12}-\lambda\,x_2x_{23}\\
                       -\lambda\,x_3x_{13}+\lambda\,x_3x_{23}\end{array}\right.$}
\end{pspicture}
\caption{The inclusion chain map $\phi:\KK_*(\B,\FF)\rightarrow\CC_*(\B,\FF)$:
the Boolean lattice
of rank $3$ (left) is marked with the Khovanov signage $\ve(x<_cy)$; the saturated
chains $\x$ starting at $x$ ($=0$ in this example) and finishing at $1$ are
marked (middle) with the resulting $\ve_\x$, and the image (right) of $\lambda\in\FF(x)$.}
\end{figure}

\begin{proposition}
$\phi:\KK_n(\B,\FF)\rightarrow\CC_n(\B,\FF)$ is a chain map.
\end{proposition}

\begin{proof}
Is accomplished by a brute force comparison of the maps $\phi d$ and $d \phi$
(where the $d$'s are the differentials in $\KK_*$ and $\CC_*$ respectively).
Let $\lambda\in\FF(x)\subset\KK_n$, and $x_1,\ldots,x_{r-n}$
be the $r-n$ elements of $\B$ covering $x$. Then,
$$
\xymatrix{
\lambda\ar@{|->}[r]^-{d}
&{\displaystyle \sum_{j=1}^{r-n}\ve(xx_j)\,\FF_{x}^{x_j}(\lambda) \ar@{|->}[r]^-{\phi}}
&{\displaystyle \sum_{j=1}^{r-n}\ve(xx_j)\sum_\x\ve(x_j\ldots x_{j_i}<_c 1)
\,\FF_{x}^{x_j}(\lambda)x_j\ldots x_{j_i}}},
$$
with the second summation over the saturated sequences $\x$ from $x_j$ to $1$.
On the other hand,
$$
\xymatrix{
\lambda\ar@{|->}[r]^-{\phi}
&{\displaystyle \sum_{j=1}^{r-n}\ve(xx_j)\sum_\x\ve(x_j\ldots x_{j_i}<_c 1)
\,\lambda\,xx_j\ldots x_{j_i}}},
$$
with again the second sum over the saturated chains $\x$ from $x_j$ to $1$.
In the image of this under the differential $d$ of the complex $\CC_*$, 
each
of the $r-n$ terms contributes a term of the form
$\ve(xx_j)\sum_\x\ve(x_j\ldots x_{j_i}<_c 1)
\,\FF_{x}^{x_j}(\lambda)x_j\ldots x_{j_i}$, obtained by dropping the $x$ from
the chain $xx_j\ldots x_{j_i}$. All the other terms have the form
\begin{equation}\label{equation100}
\ve(xx_j)\ve(x_j\ldots x_{j_i}<_c 1)(-1)^{k}
\lambda\,xx_j\ldots \wh{x}_k\ldots x_{j_i},
\end{equation}
for $j\leq k\leq j_i$, and where 
$\ve(xx_j)\ve(x_j\ldots x_{j_i}<_c 1)=\ve(xx_j\ldots x_{j_i}<_c 1)$.
The proof is thus completed by showing that all these terms cancel.
As already observed, for any chain $x_{k-1}<_c x_k<_c x_{k+1}$ in $\B$ there is a unique
$y_k\not=x_k$ with $x_{k-1}<_c y_k<_c x_{k+1}$, 
and $\ve(x_{k-1}x_kx_{k+1})=-\ve(x_{k-1}y_kx_{k+1})$. Thus, there is a 
matching term to (\ref{equation100}), indexed by 
$xx_j\ldots \wh{y}_k\ldots x_{j_i}$, and otherwise identical in all respects
except for having opposite sign. This completes the proof.
\qed
\end{proof}

We now bring in the decomposition $\B=\B_0\bigcup_f\B_1$ of the Boolean
lattice of Example \ref{example:glueings} for $\ell=1$. 
Notice that if $x\in\B_1$ and $\x$ is a sequence (saturated or not)
starting at $x$, then $\x$ is completely contained in the sublattice $\B_1$.
Thus in particular, when $\lambda\in\FF(x)$,
we have that $\phi(\lambda)$ is in the subcomplex 
$\CC_*(\B_1,\FF_1)\subset\CC_*(\B,\FF)$, and
so $\phi\,\KK_*(\B_1,\FF_1)\subset\CC_*(\B_1,\FF_1)$. We therefore
have an induced map of complexes,
$$
\phi':\KK_{*-1}(\B_0,\FF_0)=\frac{\KK_*(\B,\FF)}{\KK_*(\B_1,\FF_1)}
\rightarrow
\frac{\CC_*(\B,\FF)}{\CC_*(\B_1,\FF_1)}=Q_*.
$$

\begin{lemma}\label{lem:phiprime}
Let $\pi:Q_* \ra \CC_{*-1}(\B_0,\FF_0)$ be the map 
defined in Section 3 by equation (\ref{eq:pi}) and $\phi,\phi'$ as above. Then
the following diagram of chain maps commutes.
$$
\xymatrix{
\KK_{*-1}(\B_0, \FF_0)\ar[r]^-{\phi^\prime} \ar[dr]_-{\phi} 
& Q_* \ar[d]^-{\pi}\\
& \CC_{*-1}(\B_0,\FF_0)
}
$$
\end{lemma}

Note that the $\phi$ that appears in the diagram is the $\phi$ associated
to the sublattice $\B_0$ (not $\B$). 

\begin{proof}
Let $x\in\B_0$ and $S$ be the set of all saturated sequences 
$\x=x_1\ldots x_jy_1\ldots y_{n-j}1=x_1<_c\ldots <_cx_j<_cy_1<_c\ldots <_c y_{n-j}<_c1$
in $\B$ with the $x_i\in\B_0$,
$y_i\in\B_1$ and $x_1=x$. Let $S'\subset S$ consist of those
saturated sequences of the form $x_1\ldots x_n1$, where the $x_i\in\B_0$,
$x_1=x$ and $x_n=1_0$,
the unique maximal element of $\B_0$. Then, for $\lambda\in\FF(x)$ we have
$$
\xymatrix{
\lambda\ar@{|->}[r]^-{\phi'}
&{\displaystyle \sum_{\x\in S}^{}\ve_\x
\lambda\,\x^\circ \ar@{|->}[r]^-{\pi}}
&{\displaystyle \sum_{\x\in S'}^{}\ve_\x\lambda\,x_1\ldots x_{n-1}}
}.
$$
Now, the $\ve_\x$ that appears on
the righthand side above satisfies
$\ve_\x=\ve(x_1<_c\cdots <_cx_{n-1}<_c 1_0<_c 1)=\ve(x_1<_c\cdots<_c 1_0)\ve(1_0<_c1)$,
which in turn is just $\ve(x_1<_c\cdots<_c 1_0)$, as $\ve(1_0<_c 1)=1$.
\qed
\end{proof}

In particular we have a commuting diagram in homology: $\phi_*=\pi_*\phi'_*$.
We now have everything we need for the,

\begin{proof}[of the Main Theorem]
The short exact sequences in \S\ref{section:les} and \S\ref{section:cubecomplex}
can be assembled into a diagram,
$$
\xymatrix{
0 \ar[r] 
&  \cK_{*}(\bB_1, \FF_1) \ar[r]  \ar[d]_\phi
& \kbc * \ar[r]  \ar[d]_\phi
& \cK_{*-1}(\bB_0, \FF_0) \ar[r]  \ar[d]^{\phi^\prime} 
&0 \\
0 \ar[r]
&  \cC_{*}(\bB_1, \FF_1) \ar[r]
& \cbc *  \ar[r]
&  Q_*  \ar[r]   & 0 
}
$$
where by definition, $\phi'$ is the map making the righthand square commute, while it
is easy to check that the left-hand square commutes. By the functorality of the long
exact sequence in homology, 
we have the following commutative diagram 
$$
\xymatrix{
\cdots \ar[r]  
& {\scriptstyle H^\diamond_{n}(\B_0, \FF_0)} \ar[r]^\delta \ar[d]^-{\phi^\prime_*} 
& {\scriptstyle H^\diamond_n(\B_1, \FF_1)} \ar[r]  \ar[d]^-{\phi_*} 
& {\scriptstyle H^\diamond_n(\B, \FF)} \ar[r] \ar[d]^-{\phi_*} 
& {\scriptstyle H^\diamond_{n-1}(\B_0, \FF_0)} \ar[r]^-\delta \ar[d]^-{\phi^\prime_*} 
& {\scriptstyle H^\diamond_{n-1}(\B_1, \FF_1)} \ar[r] \ar[d]^-{\phi_*} & \cdots \\
\cdots \ar[r] 
& {\scriptstyle H_{n+1}(Q_*)} \ar[r]^-\delta 
& {\scriptstyle H_n({\bB_1}, {\FF_1})} \ar[r]  
& {\scriptstyle H_n (\bB, c)} \ar[r]  
& {\scriptstyle H_n(Q_*)} \ar[r]^-\delta 
& {\scriptstyle H_{n-1}( {\bB_1}, {\FF_1})} \ar[r]  
& \cdots
}
$$
with exact rows. The proof then proceeds by induction on the rank, noting that
the result is obviously true for Boolean lattices of rank $1$. 
If $\B$ is rank $r+1$ then both $\B_0$ and $\B_1$ are rank $r$, so assuming the
result for rank $r$ gives that the second and fifth vertical maps in the above diagram are 
isomorphisms. Furthermore, the first and fourth maps are also isomorphisms: Lemma
\ref{lem:phiprime} gives that the $\phi'_*=\pi_*^{-1}\phi_*$, where $\phi$ is again an isomorphism 
because $\B_0$ has rank $r$, and $\pi$ is an isomorphism by Proposition \ref{prop:pi}.
By the $5$-lemma, the middle
map is thus an isomorphism too.
\qed
\end{proof}

The main theorem can be strengthened somewhat: if $P$ is a poset, then 
call $\FF:P\rightarrow\rmod$ a \emph{colouring by projectives\/} if $\FF(x)$
is a projective module for all $x\in P$. As 
the direct sum of projectives is projective, and
a quasi-isomorphism between bounded below chain complexes of projectives is
a homotopy equivalence (see, eg: \cite{Wiebel94}*{\S 10.4}), we get that,

\begin{corollary}
If $\FF$ is a colouring by projectives then $\phi:\KK_n(\B,\FF)\rightarrow\CC_n(\B,\FF)$ is 
a homotopy equivalence.
\end{corollary}

In particular, as vector spaces are projective we have,

\begin{corollary}
If the ground ring $R$ of the colouring $\FF:\B\rightarrow\rmod$
is a field, then $\phi:\KK_n(\B,\FF)\rightarrow\CC_n(\B,\FF)$ is
a homotopy equivalence.
\end{corollary}

\section{Normalisation for link homology}\label{sec:normalisation}
For the motivating example, namely the Khovanov colouring of a Boolean
lattice associated to a link diagram, the modules are in fact graded
and in order to obtain an invariant result some shifts are
required. We record these shifts here in order to minimize the potential
confusion arising from our grading conventions. 

Let $V$ be the graded Frobenius algebra used in the construction of
Khovanov homology and let $(\B,\FF)$ be the Boolean lattice associated
to a given link diagram coloured with the Khovanov colouring of Example \ref{ex:khovanov}. The
grading on $V$ induces an internal grading on the associated
complex. Using the convention that $(W_{*,*}[a,b])_{i,j}= W_{i-a,j-b}$, the shifted complex we wish to consider is $\tilde{\cS}_{*,*}(\B,\FF)= \cS_{*,*}(\B,\FF)[-n_+, n_+ - 2n_-]$, i.e.
$$
\tilde{\cS}_{i,j}(\B,\FF) = \cS_{i+n_+, j - n_+ + 2n_-}(\B,\FF)
$$
where $n_+$ and $n_-$ are the number of positive and negative crossings of the (oriented) diagram. The homology  $\tilde{H}_{*,*}(\B,\FF)$ is then a bigraded link invariant. To compare with the more usual
grading in Khovanov homology we have $KH^{i,j}\cong \tilde{H}_{-i,j}$.




\end{document}